\documentclass[10pt,a4paper]{article}

\usepackage[utf8]{inputenc}
\usepackage[english]{babel}
\usepackage{color, tikz}
\usepackage{pslatex}
\usepackage{authblk} 
\usepackage{amsthm,amsmath,amssymb}
\usepackage{mathrsfs,bbm}
\usepackage{cite}
\usepackage[
  hmarginratio={1:1},     
  vmarginratio={1:1},     
  textwidth=15cm,        
  textheight=21cm,
  heightrounded,          
]{geometry}
\usepackage{hyperref}

\makeatletter

\definecolor{lightred}{rgb}{1,0.87,0.87}
\definecolor{lightgreen}{rgb}{0.8,1,0.8}
\definecolor{lightblue}{rgb}{0.8,0.8,1}
\hypersetup{
  linkbordercolor=lightred,
  citebordercolor=lightblue,
  urlbordercolor=lightgreen}

\theoremstyle{plain}
\newtheorem{theorem}{Theorem}
\newtheorem{corollary}[theorem]{Corollary}
\newtheorem{lemma}[theorem]{Lemma}
\newtheorem{proposition}[theorem]{Proposition}
\theoremstyle{definition}
\newtheorem{definition}[theorem]{Definition}
\newtheorem{remark}[theorem]{Remark}

\numberwithin{equation}{section}
\numberwithin{theorem}{section}

\@ifundefined{leqslant}{}{\let\le=\leqslant \let\leq=\le}
\@ifundefined{geqslant}{}{\let\ge=\geqslant \let\geq=\ge}
\@ifundefined{varnothing}{}{\let\emptyset=\varnothing}

\newcommand{\M}{S_\mu}

\newcommand{\G}{{\mathcal{G}}}
\newcommand{\K}{{\mathcal K}}
\newcommand{\R}{\mathbb{R}}
\newcommand{\N}{\mathbb{N}}

\newcommand{\pa}{\partial}

\newcommand{\intd}{\,\mathrm{d}}
\newcommand{\Id}{{\mathbbm 1}}

\DeclareMathOperator{\supp}{supp}
\DeclareMathOperator{\e}{e}

\DeclareMathOperator{\morse}{m}
\DeclareMathOperator{\linspan}{span}
\DeclareMathOperator{\SO}{SO}
\DeclareMathOperator{\card}{card}

\newcommand{\dd}{\@ifstar{\dd@txt}{\dd@frac}}
\newcommand{\dd@frac}[2]{
  \frac{{\operator@font d}#1}{{\operator@font d}#2}}
\newcommand{\dd@txt}[2]{
  {\operator@font d}#1 / {\operator@font d}#2}
\newcommand{\incident}{\succ}
\newcommand{\edge}{{\mathrm e}} 
\newcommand{\vv}{{\mathrm{v}}} 

\newcommand{\FIXME}{\@ifstar{\FIXMEdel}{\FIXMEadd}}
\newcommand{\FIXMEadd}[1]{\textcolor{red}{#1}}
\definecolor{FIXMEdel}{gray}{0.7}
\newcommand{\FIXMEdel}[1]{\textcolor{FIXMEdel}{\small #1}}

\newcommand{\intervalcc}[1]{\mathopen[#1\mathclose]}
\newcommand{\intervalco}[1]{\mathopen[#1\mathclose)}
\newcommand{\intervaloc}[1]{\mathopen(#1\mathclose]}
\newcommand{\intervaloo}[1]{\mathopen(#1\mathclose)}

\newcommand{\bigintervaloo}[1]{\bigl(#1\bigr)}

\makeatother

\author{Pablo Carrillo\footnote{pablo.carrillo-martinez@univ-fcomte.fr}}  \affil{Universit\'e de Franche-Comt\'e, CNRS, LmB (UMR 6623), F-25000 Besan\c{c}on, France}

\author{Damien Galant \footnote{damien.galant@umons.ac.be}} \affil{ Université Polytechnique Hauts-de-France, INSA Hauts-de-France,
CERAMATHS - Laboratoire de Matériaux Céramiques et de Mathématiques, F-59313 Valenciennes, France
and Département de Mathématiques, Université de Mons, Place du Parc, 20, B-7000 Mons, Belgium
and F.R.S. - FNRS}

\author{Louis Jeanjean\footnote{louis.jeanjean@univ-fcomte.fr}}  \affil{Universit\'e de Franche-Comt\'e, CNRS, LmB (UMR 6623), F-25000 Besan\c{c}on, France}

\author{Christophe Troestler\footnote{christophe.troestler@umons.ac.be}}  \affil{Département de Mathématiques, Université de Mons, Place du Parc, 20, B-7000
Mons, Belgium}

\title{Infinitely many normalized solutions of $L^2$-supercritical NLS equations on noncompact metric graphs with localized nonlinearities}
\date{}

\begin{document}

\maketitle

\begin{abstract}
  We consider the existence of solutions for nonlinear Schr\"odinger equations on noncompact metric graphs with localized nonlinearities.  In the $L^2$-supercritical regime, we establish the existence of infinitely many solutions for any prescribed mass.
\end{abstract}

\medskip

{\small \noindent KeyWords:
  Nonlinear Schr\"odinger equations; $L^2$-supercritical;
  noncompact metric graph; infinitely many solutions.}

\section{Introduction and main results}\label{intro}

Throughout the paper we assume that $\G$ is a noncompact
metric graph which satisfies:
\begin{equation}\label{HG}
  \G \text{ has a finite number of edges and vertices,
    a non trivial compact core } \K
  \text { and at least one half-line.}
\end{equation}
The notion of metric graph is detailed in~\cite{BeKu}.  We
  recall that if $\G$ is a metric graph with a finite number of edges
  and vertices, its compact core $\K$ is defined as the metric
  sub-graph of $\G$ consisting of all the bounded edges of $\G$
  (see~\cite{AdSeTi16, SeTeJDE16}). 

The paper is devoted to the existence of infinitely many solutions, sometimes called \textit{bound states}, of prescribed mass for the $L^2$-supercritical nonlinear Schr\"odinger (NLS) equation with localized nonlinearities on~$\G$
\begin{equation}\label{Eq: stationnaryNLS}
-u''+\lambda u=\kappa(x)|u|^{p-2}u,
\end{equation}
coupled with the Kirchhoff conditions at the vertices, see \eqref{1.2} below. Here $\lambda \in \R$ appears as a Lagrange multiplier, $p>6$, $\G$ satisfies \eqref{HG} and $\kappa$ is the characteristic function of the compact core $\K$ of $\G$.  

\medbreak
    There are several reasons coming from physics to consider Schr\"odinger equations on metric graphs.
    For instance, the so-called ``quantum graphs''
    (namely, metric graphs equipped
    with an Hamiltonian operator
    coupled with vertex conditions)
    have been introduced to model
    quantum systems having ``uni-dimensional features''.
    Works of Hückel \cite{Huckel} in the 1930s
    and then Ruedenberg-Scherr \cite{RueSch} in the 1950s
    show how the energy levels of some molecules
    correspond to the spectra
    of the Laplacian on metric graphs associated
    with the molecular structure.
    Nowadays, the study of quantum graphs
    is a vast and active field:
    we refer to~\cite{BeKu} and the references therein
    for an overview of this domain.
    
    Regarding \emph{nonlinear} Schrödinger equations
    on metric graphs,
    they have attracted much attention
    over the last few decades, as can be seen in the
    survey papers \cite{AdSeTi172,KaNoPe,No}.
    
    Remarkably, the study of NLS appears both
    in the study of \emph{matter-wave solitons}
    (as those appearing in Bose-Einstein condensates)
    and of \emph{optical solitons} (that can be realized
    in optical fibers, for instance).
    We refer to \cite[Preface]{KA}
    for a further discussion on the similarity
    between those two settings.
    In both cases, studying how the shape of underlying
    ``networks'' affects the solitary states
    is a very natural, and usually delicate, question.
    In nonlinear optics, one may create complex networks
    by connecting optical fibers.
    As for matter-wave solitons, their study in domains
    having a complex topology is closely
    related to the emerging field of \emph{atomtronics},
    which aims to realize circuits of ultracold matter
    exhibiting quantum effects. We do not attempt
    to provide more details about this fascinating subject
    here and refer to \cite{AmAnBoBrKwMivonKl}
    and to \cite[Section~1]{AdSeTi172}
    for further information.

    The localization of the nonlinearity
    appears when modeling a network made
    of optical fibers of two kinds, one kind having
    a much stronger nonlinear effect than the other.
    As a first approximation, one may thus consider
    that all fibers in the compact core have the same
    nonlinear effect and that all the remaining fibers
    do not have any nonlinear effect.
    From the point of view of physics,
    the richness of this model lies in the interplay
    between the nonlinearity and the diffusive
    effects (usually leading to scattering).
    We refer to~\cite{GnSmDe} (see also \cite{No, Te})
    for further discussion on these aspects.

\medbreak
Solutions to \eqref{Eq: stationnaryNLS} with prescribed mass, often referred to as {\textit{normalized solutions}}, correspond to critical points of the NLS energy functional $E(\cdot,\G)\colon H^1(\G)\to \R$ defined by
\begin{equation}\label{Eq: E(u)}
  E(u,\G) = \frac{1}{2}\int_\G|u'|^2 \intd x
  -\frac{1}{p} \int_\K |u|^p \intd x,
\end{equation}
under the corresponding mass constraint
\begin{equation}\label{Eq: mass constraint}
  \int_\G|u|^2 \intd x = \mu > 0.
\end{equation}
It is standard to show that $E(\cdot, \G)$ is of class $C^2$ on $H^1(\G)$.
Note that  solutions to \eqref{Eq: stationnaryNLS} provide standing waves of the time-dependent focusing NLS on $\G$,
\[
  \mathrm{i}\, \pa_t \psi(t,x)
  = -\partial_{xx} \psi(t,x)-\kappa(x)|\psi(t,x)|^{p-2} \psi(t,x),
\]
via the ansatz $\psi(t,x) = \e^{\mathrm{i} \lambda t}u(x)$. The
constraint \eqref{Eq: mass constraint}
is meaningful from a dynamics perspective as
the mass (or charge), as well as the energy,
is conserved by the NLS flow.
This constraint is also very natural from
the point of view of physics. For instance, when studying
Bose-Einstein condensates, the $L^2$-norm is related
to the quantity of matter inside the system under study
(see e.g.~\cite[Section~1]{AdSeTi172}).

Recently, much effort has been devoted to establish the
existence of normalized solutions of NLS on metric graphs, in the
$L^2$\emph{-subcritical} (i.e., $p\in(2,6)$) or $L^2$\emph{-critical
  regimes} (i.e., $p=6$). In these two regimes, the energy functional
$E(\cdot, \G)$ is bounded from below and coercive on the mass
constraint.  A relevant notion is then the one of ground states,
namely of solutions which minimize the energy functional on the
constraint. For the existence of ground state solutions,
the reader can consult
\cite{AdBoDo22,AdCaFiNo, AdSeTi15,AdSeTi16,AdSeTi17,NoPe,PiSo} for
noncompact graphs $\G$, and \cite{CaDoSe, Do18} for compact
ones; some studies are also conducted on the existence of
local minimizers, see e.g.~\cite{AdSeTi18,PiSoVe}.


Regarding problems with a localized nonlinearity as in \eqref{Eq:
  stationnaryNLS}, existence and non-existence of ground state
solutions was discussed in \cite{Te} and of bound state solutions in
\cite{SeTeNonl16} for the $L^2$-subcritical case. We refer to
\cite{DoTe19,DoTe20} for the same problem on the $L^2$-critical
case. Moreover, in the $L^2$-subcritical regime, one may
obtain the existence of multiple bound states with negative
energy levels by applying genus theory both in the compact case as in
\cite{Do18} and in the noncompact case with localised nonlinearities
as in \cite{SeTeJDE16}.

However, in the $L^2$-supercritical regime on general metric graphs,
i.e., when $p>6$, the energy functional $E(\cdot, \G)$ is always
unbounded from below.  Moreover, due to the fact that graphs are not
scale invariant, the techniques based on scalings, usually employed in
the Euclidean setting and related to the validity of a Pohozaev
identity (see \cite{Je97} or \cite{BaSo17,BaSo18, IkTa19, SoJDE20,
  SoJFA20}), do not work. These two features make the search for
normalized solutions in the $L^2$-supercritical regime
delicate. Recently, in \cite{ChJeSo}, this issue was considered on
compact metric graphs for which the existence of a
non-constant solution was proved for small values of
$\mu >0$. In \cite{BoChJeSo23}, the case of a noncompact graph
with a nonlinearity acting only on its compact core was considered.
For any mass the existence of at least one positive solution
to \eqref{Eq: stationnaryNLS} was obtained.  Our aim here is to show
that, under exactly the same assumptions as in
\cite{BoChJeSo23}, the existence of infinitely many, possibly
sign-changing, solutions can be obtained for an arbitrary mass.

\subsection*{Basic notations and main result}

For any graph, we write
$\G = (\mathcal{E}, \mathcal{V})$, where $\mathcal{E}$ is the set of
edges and $\mathcal{V}$ is the set of vertices.  Each bounded edge
$\edge$ is identified with a closed bounded interval
$I_\edge = [0,\ell_\edge]$ (where $\ell_\edge$ is the length
of $\edge$), while each unbounded edge is identified with a closed
half-line $I_\edge = \intervalco{0,+\infty}$.
The length of the shortest path between points provides
  a natural metric (whence a topology and a Borel structure)
  on $\G$.
A function $u: \G \to \R$ is identified with a vector of functions
$\{u_{\edge}\}_{\edge \in \mathcal{E}}$, where each $u_{\edge}$ is
defined on the corresponding interval $I_\edge$ such that
$u|_{\edge}=u_{\edge}$.  Endowing each edge with the Lebesgue
measure, one can define $\int_\G u(x) \intd x$ and the space
$L^p(\G)$ in a natural way, with norm
\begin{equation*}
  \|u\|_{L^p(\G)}^p
  = \sum_{\edge \in \mathcal{E}} \|u_\edge\|_{L^p(\edge)}^p.    
\end{equation*}
The Sobolev space $H^1(\G)$ consists of the set of continuous
functions $u: \G \to \R$ such that
$u_{\edge} \in H^1(\mathring{e})$ for every edge
$\edge$; the norm
in $H^1(\G)$ is defined as
\begin{equation*}
  \|u\|_{H^1(\G)}^2
  = \sum_{\edge \in \mathcal{E}} \|u_{\edge}'\|_{L^2(\edge)}^2
  + \|u_{\edge}\|_{L^2(\edge)}^2.
\end{equation*}
More details can be found in \cite{AdSeTi15,AdSeTi16,BeKu}.

We shall study the existence of critical points of the functional $E(\cdot,\mathcal{G})\colon H^1(\G)\to \R$ constrained on the $L^2$-sphere
\begin{equation*}
  H_{\mu}^1(\mathcal{G})
  := \Bigl\{u\in H^1(\mathcal{G}) \Bigm|
  \int_{\mathcal{G}}|u|^2\intd x=\mu \Bigr\}.
\end{equation*}
If 
 $u\in H^1_{\mu}(\mathcal{G})$ is such a critical point, it is standard to show that there exists a Lagrange multiplier $\lambda\in \mathbb{R}$ such that $u$ satisfies the following problem:
\begin{equation}\label{1.2}
  \begin{cases}
    -u''+\lambda u=\kappa(x)|u|^{p-2}u
    &\text{on every edge } \edge \in \mathcal{E},\\[1\jot]
    \displaystyle
    \sum_{\edge \incident \vv} \dd{u_{\edge}}{x}(\vv)=0
    &\text{at every vertex } \vv \in \mathcal{V},
  \end{cases}
\end{equation}
where $\edge \incident \vv$ means that the edge $\edge$ is incident at
$\vv$, and the notation $\dd*{u_{\edge}}{x}(\vv)$ stands for
$u'_{\edge}(0)$ or $-u'_{\edge}(\ell_{\edge})$, according to whether
the vertex $\vv$ is identified with $0$ or $\ell_{\edge}$ (namely, the
sum involves the derivatives away from the vertex $\vv$). The
second equation is the so-called \emph{Kirchhoff boundary condition}.
\medskip

Our main result is the following :

\begin{theorem}\label{thm: Whole ex}
  Let $\G$ be any metric graph satisfying Assumption \eqref{HG} and
  $p>6$. Then, for any \mbox{$\mu >0$}, Problem~\eqref{1.2} with
  the mass constraint \eqref{Eq: mass constraint} has infinitely many
  distinct solutions. Moreover, these solutions are associated to
  positive Lagrange multipliers and correspond to a sequence of
  critical points of the functional $E(\cdot\,,\mathcal{G})$
  constrained on $H_{\mu}^1(\mathcal{G})$
whose levels go to $+ \infty$.
\end{theorem}
In the derivation of the results of \cite{BoChJeSo23,ChJeSo}, a
central difficulty was the lack a priori bounds on the Palais-Smale
sequences for $E(\cdot\,,\mathcal{G})$ constrained to
$H_{\mu}^1(\mathcal{G})$. To overcome this difficulty an approach by
approximation was developed.  It consists in considering
the family of functionals
$E_{\rho}(\cdot,\mathcal{G}): H^1(\mathcal{G}) \to \mathbb{R}$ given
by
\begin{eqnarray}\label{para func}
  E_{\rho}(u,\mathcal{G})
  := \frac{1}{2}\int_{\mathcal{G}}|u'|^2 \intd x
  -\frac{\rho }{p} \int_{\mathcal{K}}|u|^p \intd x,
  \qquad \forall u\in H^1(\mathcal{G}), \
  \forall \rho\in \left[\frac12, 1 \right] .
\end{eqnarray}
We shall also proceed this way.  Clearly a critical point of
$E_{\rho}(\cdot, \G)$ constrained to $H_{\mu}^1(\mathcal{G})$ is
a solution to
\begin{equation}\label{1.2LL}
  \begin{cases}
    -u''+\lambda u= \rho\kappa(x)|u|^{p-2}u
    &\text{on every edge } \edge \in \mathcal{E},\\[1\jot]
    \displaystyle
    \sum\limits_{\edge \incident \vv} \dd{u_{\edge}}{x}(\vv)=0
    &\text{at every vertex } \vv \in \mathcal{V},
  \end{cases}
\end{equation}
where $\lambda$ is the associated Lagrange multiplier. Denoting by
$\morse(u)$ the Morse index of a solution $u \in H^1_{\mu}(\G)$ to
\eqref{1.2LL},
we establish 

\begin{theorem}\label{thm: main exL}
  Let $\G$ satisfy Assumption~\eqref{HG} and $p>6$. For any $\mu >0$
  there exists $N_0 \in \N$ such that for
  almost every $\rho\in [1/2,1]$, there exist sequences
  of Lagrange multipliers
  $\{\lambda_{\rho}^N\}_{N = N_0}^{\infty} \subset \R^+$ and solutions
  $\{u_{\rho}^N \}_{N = N_0}^{\infty} \subset H^1_{\mu}(\G)$ to
  \begin{equation}\label{Eq:pb rho}
    \begin{cases}
      -(u_\rho^N)'' + \lambda_\rho^N u_\rho^N
      = \rho \kappa(x) |u_\rho^N|^{p-2} \, u_{\rho}^N
      &\text{on every edge } \edge \in \mathcal{E}, \\[1.5\jot]
      \displaystyle
      \sum_{\edge \incident \vv} \dd{u_\rho^N}{x}(\vv) = 0
      &\text{at every vertex } \vv \in \mathcal{V}.
    \end{cases}
  \end{equation}
  In addition, $c_{\rho}^N := E_{\rho}(u_{\rho}^N, \G)
  \xrightarrow[N \to \infty]{} + \infty$
  uniformly w.r.t.\ $\rho \in \intervalcc{1/2, 1}$
  and $\morse(u_{\rho}^N) \le N+1$.
\end{theorem}


To derive Theorem \ref{thm: Whole ex} from Theorem \ref{thm: main
  exL}, one considers for every fixed $\mu >0$ and every fixed
$N \geq N_0$, a sequence $\{u_{\rho_n}^N\}_{n=1}^{\infty}$ of solutions
to \eqref{Eq:pb rho} where $\rho_n \to 1^-$ and shows that it
converges to some $u^N \in H^1_{\mu}(\G)$. Such
$u^N \in H^1_{\mu}(\G)$ will be a solution to \eqref{Eq: mass
  constraint}--\eqref{1.2}. The point here is to show that the sequence
is bounded which in turn is equivalent to showing that the sequence
$\{\lambda_{\rho_n}^N\}_{n=1}^{\infty} \subset \R$ is bounded.  In
\cite{BoChJeSo23,ChJeSo} this step was done through a blow-up analysis
taking advantage that $u_{\rho_n}^N \in H^1_{\mu}(\G)$ were positive
functions. A more general blow-up analysis, in particular for possibly
sign-changing solutions, was subsequently performed in
\cite{CaGaJeTr}.  A consequence of this blow-up analysis
(see \cite[Corollary~1.4]{CaGaJeTr}),
stated here under our notation in Lemma~\ref{info-norms}, guarantees
the boundedness of the sequence of
$\{\lambda_{\rho_n}^N\}_{n=1}^{\infty} \subset \R$ thanks to
the boundedness of the Morse index of the solutions
$u_{\rho_n}^N \in H^1_{\mu}(\G)$.

\smallskip

Now let us turn to the proof of Theorem \ref{thm: main exL}.
It relies on an
abstract result~\cite[Theorem~1.12]{BoChJeSo2} which we recall here as
Theorem~\ref{Objective1}.  Used on our family
$E_{\rho}(\cdot,\mathcal{G}): H^1(\mathcal{G}) \to \mathbb{R}$, it will
guarantee that, for any $\mu >0$ and any $N \in \N$, under some
geometric conditions, the functional $E_{\rho}(\cdot, \G)$ admits, for
almost every $\rho \in [1/2,1]$, a \emph{bounded} Palais-Smale
sequence $\{u_{\rho, n}^N\}_{n =1}^{\infty}$
at level $c_{\rho}^N$ which has an ``approximate constrained
Morse index at most~$N$.''

\smallskip

To be more specific, our strategy to prove Theorem \ref{thm:
  main exL} is the following.  First we show that the
geometrical assumptions on $E_{\rho}(\cdot,\mathcal{G})$,
$\rho \in [1/2,1]$ are satisfied. Second, we check that the
Palais-Smale sequences provided by the application of Theorem
\ref{Objective1} converge. Finally, we observe that this process
guarantees the existence of infinitely many distinct solutions
$u_{\rho}^N \in H^1_{\mu}(\G)$ since $c_{\rho}^N \to + \infty$ as
$N \to + \infty$.
Let us now provide more information on the first two steps.

\smallskip

The fact that
the mentioned geometric assumptions hold is established in
Proposition~\ref{Prop: inftylevels}.  Proving this proposition
is a central part of the paper and for this we are indebted to ideas
from \cite{BaVa,BaLiLi,PiVe}. Our proof of Proposition \ref{Prop:
  inftylevels} uses the assumption that $\G$ has at least one
half-line, and it's unclear whether a similar result would
hold if the graph were compact. As a result the noncompactness of $\G$
appears to be essential in the derivation of
Theorem~\ref{thm: main exL}, see also Remark~\ref{Gcompact}.

\smallskip

Regarding the convergence of the bounded Palais-Smale
sequences $\{u_{\rho,n}^N\}_{n =1}^{\infty}$ provided by the
application of Theorem \ref{Objective1},
an essential element of
  the argument is to establish that the associated sequence of
\textit{almost Lagrange multipliers}
$\{\lambda_{\rho, n}^N\}_{n =1}^{\infty}$ (see
  page~\pageref{def-almost-Lagrange}) converges, up to a subsequence,
to a positive $\lambda_{\rho}^N \in \R$. This is done in two steps.
First, making use of the Morse type information carried by the
sequence $\{u_{\rho,n}^N\}_{n =1}^{\infty}$, we show that
$\lambda_{\rho}^N <0$ is impossible.  Here again we use the assumption
that our graph contains one half-line, see the proof of Lemma
\ref{lemma:L-eigenvalue}. Second, to show that
$\lambda_{\rho}^N \neq 0$ requires a specific treatment.  In
\cite{BoChJeSo23,ChJeSo} we were dealing with Palais-Smale sequences
consisting of non-negative functions and thus their weak
limits (which are solutions to \eqref{1.2LL} with possibly a smaller
$L^2$ norm than $\sqrt{\mu})$ were also non-negative.  It was then
rather direct to show that $\lambda_{\rho}^N >0$: see
\cite[Remark 1.2]{ChJeSo} in the case of a compact graph, or
\cite[Proof of Proposition 1.5]{BoChJeSo23} in the case of a
noncompact graph with a localized nonlinearity. In
our problem the weak limits are likely to be sign-changing.
In general, there may exist nonzero solutions with a vanishing Lagrange multiplier,
as was already observed in \cite[Section 4]{SeTeNonl16}.
For a simple example (taken from \cite[Theorem 4.2 and Remark 4.6]{SeTeNonl16}),
consider the tadpole graph shown in Figure~\ref{fig:tadpole}.

\begin{figure}[h]
    \centering
	\begin{tikzpicture}
        \draw(0, 0) circle (.5);
		\node at (0.5, 0) [circle, draw, fill, inner sep=0pt, minimum size=0.5*width("k")] (1) {};
        \node at (3.5, 0) [circle, inner sep=0pt, minimum size=0mm] (2) {$\scriptstyle\infty$};
		\draw [-] (1) -- (2);
	\end{tikzpicture}
    \caption{A tadpole graph}
    \label{fig:tadpole}
\end{figure}
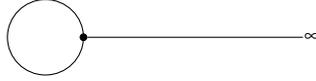

If the loop has a suitable length, one can put a sign-changing periodic solution of the equation
$-u'' = |u|^{p-2}u$ on the loop and extend it by zero on the half-line to obtain a solution of the problem
on the whole tadpole graph with a Lagrange multiplier equal to zero.
To treat general graphs we make use
of ODE techniques in a way which
we believe new in this context. Assuming that $\lambda =0$ in
\eqref{1.2LL}, we show that the $L^2$ norm of a solution
$u \in H^1_{\mu}(\G)$ goes to infinity as $E(u, \G)$ goes to infinity,
see Lemma \ref{below_estimate} for a precise statement.  This
observation leads to the conclusion that if the suspected energy
level, $c_{\rho}^N \in \R$ is sufficiently high, the case
$\lambda_{\rho}^N = 0$ cannot happen. Having obtained that
$\lambda_{\rho}^N >0$ and using that the nonlinearity is compactly
supported we obtain the convergence of our Palais-Smale sequences and
this proves Theorem \ref{thm: main exL}.

\begin{remark}\label{comparisonsub}
  Our multiplicity result Theorem \ref{thm: Whole ex} is in sharp
  contrast to what has been observed in \cite{SeTeJDE16, SeTeNonl16}
  in the mass subcritical case $p <6$.  Indeed,
  \cite[Corollary~3.8]{SeTeNonl16} shows that for a \emph{graph
    without cycle} (also called a tree), with at most one
  \emph{pendant} (see \cite{SeTeNonl16} for the terminology), there
  are no solutions to \eqref{Eq: mass constraint}--\eqref{1.2}
  when $ p \in [4, 6[$ and $\mu >0$ is small. Also, in \cite[Theorem
        1.2]{SeTeJDE16}, to obtain $k \in \N$ solutions it is
      necessary to assume that $\mu > \mu(k)$. We have no such
      limitations in Theorem \ref{thm: Whole ex}.
\end{remark}

\begin{remark}\label{localizedcreate}
  As it was already observed in \cite{SeTeJDE16} in the mass
  subcritical case, the localization of the nonlinearity on the
  non-trivial compact core is essential to our multiplicity
  results.
  Indeed, if the compact core is reduced to a point, $\G$
    is a star graph and \eqref{1.2} becomes linear.  This problem
    possesses no solution in $H^1(\G)$ regardless of the value of
    $\mu > 0$.  On the other hand, if $\G$ is an interval with two
    half-lines attached to its endpoints and the nonlinearity is not
    localized, then solutions to \eqref{Eq: mass
      constraint}--\eqref{1.2} are the same as those on $\R$, namely
    the unique symmetric positive ground state, its opposite, along
    with their translations (all of which have the same energy level).
\end{remark}

\begin{remark}\label{Gcompact}
  Let us mention that the issue of multiplicity, even the existence of
  just two non-trivial solutions, is still open for a general
    compact graph $\G$.  In \cite[Theorem~1.1]{ChJeSo} only
  one non-constant solution solution is obtained (note that there
  always exists a constant solution to \eqref{Eq: mass
    constraint}--\eqref{1.2} on a compact graph).
\end{remark}

The paper is organized as follows. In Section~\ref{Preliminaries_I} we
recall with Theorem~\ref{Objective1} the contents of
\cite[Theorem~1.12]{BoChJeSo2} and explore some of its consequences. In
particular, we show that second-order information on Palais-Smale
sequences can be used to obtain uniform bounds from below on the
associated sequences of almost Lagrange multipliers, see Lemma
\ref{lemma:Louis1}. We also derive some results, in Lemma
\ref{Pierotti-Verzini} and Theorem \ref{prop:MPcondition}, which
provide abstract conditions allowing to check that the main
assumptions of Theorem \ref{Objective1} hold.  Most of
Section~\ref{Preliminaries_II} is devoted to show that solutions
to~\eqref{1.2LL} with $\lambda =0$ have a $L^2$ norm going to infinity
as their Energy goes to infinity (see Proposition~\ref{forbid}). In
Section \ref{Proof_Th2}, we prove Proposition \ref{Prop: inftylevels}
which shows that our problem can indeed be treated by an application
of Theorem \ref{Objective1}. In Section \ref{section:strong conv} we
give the proof of Theorem \ref{thm: main exL}.  Finally, in Section
\ref{Proof_Th1} we deduce Theorem \ref{thm: Whole ex} from Theorem
\ref{thm: main exL} making use of the already mentioned blow up
analysis result from~\cite{CaGaJeTr}.

\section{An Abstract Multiplicity Result}
\label{Preliminaries_I}

In this section we recall in Theorem \ref{Objective1} the contents of
\cite[Theorem~1.12]{BoChJeSo2} and present some of its
consequences. We also establish results which permit to check the two
main hypotheses the set defined by \eqref{SMP1} must satisfy: Lemma
\ref{Pierotti-Verzini} gives a procedure to prove it is non-void and
Theorem \ref{prop:MPcondition} provides a tool to check the key strict
inequality \eqref{SMP2} appearing in Theorem
\ref{Objective1}. \smallskip

In order to state \cite[Theorem~1.12]{BoChJeSo2}  we need to recall some definitions. \smallskip

Let $\left(E,\langle \cdot, \cdot \rangle\right)$ and $\left(H,(\cdot,\cdot)\right)$ be two \emph{infinite-dimensional} Hilbert spaces and assume that $E\hookrightarrow H \hookrightarrow E'$,
with continuous injections.  For simplicity, assume that the
continuous injection $E\hookrightarrow H$ has norm at most $1$ and
identify $E$ with its image in $H$.  Set
\begin{equation*}
  \begin{cases}
    \|u\|^2=\langle u,u \rangle, &u\in E,\\
    |u|^2=(u,u), &u \in H,
  \end{cases}
\end{equation*}
and define for $\mu>0$:
\begin{equation*}
  S_\mu = \bigl\{ u \in E\bigm| |u|^2=\mu \bigr\}.
\end{equation*}
In the context of this paper, we shall have $E=H^1(\G)$ and $H=L^2(\G)$.
Clearly, $S_{\mu}$ is a smooth submanifold of $E$ of codimension
$1$. Furthermore its tangent space at a given point $u \in S_{\mu}$
can be considered as the closed subspace of codimension $1$ of $E$
given by:
\begin{equation*}
  T_u S_{\mu} = \bigl\{v \in E \bigm| (u,v) =0 \bigr\}.
\end{equation*}
In the following definition, we denote  $\|\cdot\|_*$ and
$\|\cdot\|_{**}$ the operator norm of $\mathcal{L}(E,\R)$ and of
$\mathcal{L}(E,\mathcal{L}(E,\R))$ respectively.

\begin{definition}\label{Holder-continuous}
  Let $\phi : E \rightarrow \mathbb{R}$ be a $C^2$-functional on $E$
  and $\alpha \in \intervaloc{0,1}$. We say that $\phi'$ and $\phi''$
  are \emph{$\alpha$-H\"older continuous on bounded sets} if for any
  $R>0$, one can find $M=M(R)>0$ such that, for any
  $u_1,u_2\in B(0,R)$:
  \begin{equation}\label{Holder}
    \|\phi'(u_1)-\phi'(u_2)\|_* \leq M \|u_1-u_2\|^{\alpha}, \qquad
    \|\phi''(u_1)-\phi''(u_2)\|_{**} \leq M\|u_1-u_2\|^\alpha.
  \end{equation}
\end{definition}

\begin{remark}
  \label{rem:Holder-cont}
  Note that, if $\phi''$ is $\alpha$-H\"older continuous on bounded
  sets, then $\phi'$ is Lipschitz continuous on bounded sets,
  whence also $\alpha$-H\"older continuous on bounded sets.
  
\end{remark}

\begin{definition}\label{def D}
  Let $\phi$ be a $C^2$-functional on $E$.  For any $u\in E$, we
  define the continuous bilinear map:
  \begin{equation*}
    D^2\phi(u)
    := \phi''(u) - \frac{\phi'(u)[u]}{|u|^2} (\cdot,\cdot).
  \end{equation*}
\end{definition}
Note that, if $u$ is a critical point of $\phi$ restricted to
  the sphere $S_\mu$, then $D^2 \phi(u)$, seen as a bilinear form on
  $T_u S_\mu$, is the second derivative of $\phi|_{S_\mu}$ at $u$.

\begin{definition}\label{def: app morse}
  Let $\phi$ be a $C^2$-functional on $E$. { For} any $u\in \M$ and
  $\theta \ge 0$, we define the \emph{approximate Morse index} by
  \begin{equation*}
    \tilde \morse_\theta(u)
    = \sup \bigl\{\dim L \bigm| L \text{ is a subspace of }
      T_u \M \text{ such that }
      \forall \varphi \in L \setminus \{0\},\
      D^2\phi(u)[\varphi, \varphi]<-\theta \|\varphi\|^2
      \bigr\}.
  \end{equation*}
  If $u$ is a critical point for the constrained functional
  $\phi|_{\M}$ and $\theta=0$, we say that this is the \emph{Morse
    index of $u$ as constrained critical point}.
\end{definition}

We may now formulate \cite[Theorem~1.12]{BoChJeSo2}. 
 Its derivation is based on a combination of ideas from \cite{FaGh-1994,Je99} implemented in a convenient geometric setting. 

\begin{theorem}\label{Objective1}
  Let $I\subset (0,\infty)$ be an interval and consider a family of
  $C^2$ functionals $\Phi_\rho\colon E\to \R$ of the form
  \begin{equation*}
    \Phi_\rho(u)=A(u)-\rho B(u),\qquad \rho\in I,
  \end{equation*}
  where $B(u)\geq 0$ for all $u\in E$ and
  \begin{equation}\label{hp coer}
    A(u) \to +\infty
    \text{ or } B(u) \to +\infty
    \quad \text{as } u \in E \text{ and }
    \|u\| \to +\infty.
  \end{equation}
  Suppose that, for every $\rho \in I$, $\Phi_\rho|_{S_\mu}$ is even
  and moreover that $\Phi'_\rho$ and $\Phi''_\rho$ are
  $\alpha$-H\"older continuous on bounded sets in the sense of
  Definition~\ref{Holder-continuous} for some
  $\alpha \in \intervaloc{0,1}$. Finally, suppose that there exists
  an integer $N \geq 2$ and two odd functions
  $\gamma_i : \mathbb{S}^{N-2} \to S_{\mu}$ where $i =0,1$, such
  that the set
  \begin{equation}\label{SMP1}
    \Gamma_N
    := \bigl\{ \gamma\in C([0,1]\times \mathbb{S}^{N-2}, S_\mu)
    \bigm| \forall t \in [0,1],\  \gamma(t,\cdot) \text{ is odd, }
    \gamma (0, \cdot) = \gamma_0, \text{ and }
    \gamma(1, \cdot) = \gamma_1 \bigr\}
  \end{equation}
  is non void and 
  \begin{equation}\label{SMP2}
    c_{\rho}^N
    := \inf_{\gamma\in\Gamma_N}
    \,\max_{(t,a) \in [0,1] \times \mathbb{S}^{N-2}}
    \Phi_\rho(\gamma(t,a))
    > \max_{a \in \mathbb{S}^{N-2}}
    \bigl\{\Phi_\rho (\gamma_0(a)), \Phi_\rho (\gamma_1(a))\bigr\},
    \quad \forall \rho \in I.
  \end{equation}
  Then, for almost every $\rho \in I$, there exist sequences
  $\{u_n\} \subset \M$ and $\zeta_n \to 0^+$ such that, as
  $n \to + \infty$,
  \begin{itemize}
  \item[(i)] $\Phi_\rho(u_n) \to c_{\rho}^N$;
  \item[(ii)] $\bigl\| \Phi'_\rho|_{\M}(u_n) \bigr\| \to 0$;
  \item[(iii)] $\{u_n\}$ is bounded in $E$;
  \item[(iv)] $\tilde\morse_{\zeta_n}(u_n) \leq N$.
  \end{itemize}
\end{theorem}

\begin{remark}\label{strong-convergence}
  If the sequence $\{u_n\} \subset \M$ provided by
    the previous Theorem converges to some $u_\rho \in \M$, then
  in view of points (i)--(ii),
  $u_\rho$ is a critical point of
  $\Phi_\rho|_{\M}$ at level $c_{\rho}^N$.  Let us show that the
  Morse index of $u_{\rho}$, as a
  constrained critical point, satisfies
  $\tilde\morse_{0}(u_{\rho}) \leq N$.  Assume by contradiction
  that this is not the case. Then, in view of Definition \ref{def: app
    morse}, we may assume that there exists a
  $W_0 \subset T_{u_{\rho}}\M$ with $\dim W_0 =N+1$ such that
  \begin{equation*}
    D^2 \Phi_\rho(u_{\rho})[w,w]<0
    \quad \text{for all } w \in W_0 \setminus \{0\}.
  \end{equation*}
  Since $W_0$ is of finite dimension, its unit sphere is
    compact and there exists $\theta > 0$ such that
  \begin{equation*}
    D^2 \Phi_\rho(u_{\rho})[w,w] < -\theta \|w\|^2
    \quad \text{for all } w \in W_0 \setminus\{0\}.
  \end{equation*}
  Now, from \cite[Corollary~1]{BoChJeSo2} or using directly that
  $\Phi_\rho'$ and $\Phi_\rho''$ are $\alpha$-H\"older continuous on
  bounded sets for some $\alpha \in \intervaloc{0,1}$, it follows that
  there exists $\delta>0$ small enough such that, for any $v\in\M$
  satisfying $\|v-u_{\rho}\| \leq \delta$,
  \begin{equation*}
    D^2\Phi_\rho(v)[w,w] < -\frac{\theta}{2} \|w\|^2
    \quad \text{for all } w \in W_0 \setminus \{0\}.
  \end{equation*}
  In particular, for $n$ large enough, $\|u_n - u_\rho\| \le
    \delta$ and $\zeta_n < \theta/2$ (as $\zeta_n \to 0^+$),
    so the previous inequality implies
  \begin{equation*}
    D^2\Phi_\rho(u_n)[w,w] < -\frac{\theta}{2}\|w\|^2 < -\zeta_n\|w\|^2
    \quad \text{for all } w \in W_0 \setminus \{0\} .
  \end{equation*}
  Remembering that $\dim W_0>N$ and observing
  that Theorem \ref{Objective1} (iv) directly implies that if there
  exists a subspace $W_n \subset T_{u_n}\M$ such that
  \begin{equation*}
    D^2\Phi_{\rho}(u_n)[w,w]  < - \zeta_n \|w\|^2,
    \quad \text{for all } w \in W_n \setminus \{0\},
  \end{equation*}
  then necessarily $\dim W_n \leq N$, we have reached a
  contradiction.
\end{remark}

From Theorem \ref{Objective1}  (ii)--(iii), we deduce in a standard way, see \cite[Remarks~1.3]{BoChJeSo2} or \cite[Lemma~3]{BeLi2}, that 
\begin{equation}\label{free-gradient}
  \Phi'_{\rho}(u_n) + \displaystyle \lambda_n (u_n, \cdot) \to 0
  \quad\text{ in }  E' \text{ as } n \to + \infty
\end{equation}
where we have set 
\begin{equation}\label{def-almost-Lagrange} 
  \lambda_n
  := - \displaystyle \frac{1}{\mu} \Phi'_\rho(u_n)[u_n].
\end{equation} 
We call the sequence $\{\lambda_n\} \subset \R$ defined in
\eqref{def-almost-Lagrange} the sequence of \emph{almost Lagrange
  multipliers}. \medskip

The following lemma will allow to derive information on such sequences. 
\begin{lemma}\label{lemma:Louis1}
  Let $\{u_n\} \subset S_{\mu}$, $\{\lambda_n\} \subset \mathbb{R}$
  and $\{\zeta_n\}\subset \mathbb{R}^+$ with $\zeta_n \to 0^+$. Assume
  that, for a given $M \in \N$, the following conditions hold:
  \begin{itemize}
  \item[(i)] For large enough $n \in \mathbb{N}$, all subspaces
    $W_n \subset E$ with the property
    \begin{equation}\label{LL-Hess crit*}
      \Phi''_{\rho}(u_n)[\varphi,\varphi] + \lambda_n|\varphi|^2
      < -\zeta_n \|\varphi\|^2,
      \qquad \text{for all }\, \varphi \in W_n \setminus \{0\},
    \end{equation}
    satisfy: $\dim (W_n) \leq M$.
  \item[(ii)] There exist $\lambda \in \mathbb{R}$, a subspace $Y$ of
    $E$ with $\dim (Y) \geq M+1$ and $\zeta > 0$ such that, for large enough
    $n \in \mathbb{N}$,
    \begin{equation}\label{L1*}
      \Phi''_{\rho}(u_n)[\varphi,\varphi] + \lambda |\varphi|^2
      \le -\zeta \|\varphi\|^2,
      \qquad \text{for all }\, \varphi \in Y.
    \end{equation}
  \end{itemize}
  Then $\lambda_n > \lambda$ for all large enough $n \in
  \mathbb{N}$. In particular, if (\ref{L1*}) holds for any
  $\lambda <0$, then
  $\displaystyle \liminf_{n \to \infty} \lambda_n \geq 0.$
\end{lemma}
\begin{proof}
  Suppose by contradiction that $\lambda_n \leq \lambda$ along a
  subsequence still denoted $\{\lambda_n\}$.  Keep
    denoting $\{u_n\}$ and $\{\zeta_n\}$ the corresponding subsequences.
  From \eqref{L1*} we have,
\begin{equation*}
  \Phi''_{\rho}(u_n)[\varphi,\varphi]+\lambda_n|\varphi|^2
  \le \Phi''_{\rho}(u_n)[\varphi,\varphi] + \lambda |\varphi|^2
  \le -\zeta \|\varphi\|^2
  \qquad
  \text{for all } \varphi \in Y \setminus \{0\}.
\end{equation*}
Now, since $\zeta_n\to 0^+$, there exists $n_0 \in \mathbb{N}$ such that: $\forall n \geq n_0$,
$\zeta_n < \zeta$. Thus, for an arbitrary $n \geq n_0$, we obtain 
\begin{equation*}
  \Phi''_{\rho}(u_n)[\varphi,\varphi]+\lambda_n|\varphi|^2
  \le  -\zeta_n \|\varphi\|^2,
  \qquad \text{for all } \, \varphi \in Y \setminus \{0\},
\end{equation*}
in contradiction with \eqref{LL-Hess crit*} since $ \dim  (Y) \geq M+1$. 
\end{proof}

We will now focus on deriving sufficient conditions used to verify the
two hypotheses posed on the class of paths $\Gamma_N$ in Theorem
\ref{Objective1}. The following Lemma, directly inspired by
\cite[Remark~4.5]{PiVe}, deals with the first hypothesis of showing
that the set is non void.

\begin{lemma}\label{Pierotti-Verzini}
  Let $\{u_1,\dots, u_{N-1}\} \subset \M $ and
  $\{v_1,\dots, v_{N-1}\} \subset \M$ be orthogonal families
  for the inner product $(\cdot, \cdot)$.
  Setting the odd functions
  \begin{align*}
    \gamma_0 : \mathbb{S}^{N-2} \to \M \quad \text{by} \quad 
    \gamma_0(a_1,\dots,a_{N-1}) =  \sum_{i=1}^{N-1} a_i u_i
  \end{align*}
  and
  \begin{align*}
    \gamma_1 : \mathbb{S}^{N-2} \to \M \quad \text{by} \quad 
    \gamma_1(a_1,\dots,a_{N-1}) =  \sum_{i=1}^{N-1} a_i v_i ,
  \end{align*}
  the set $\Gamma_N$ defined by~\eqref{SMP1}
  is non void.
\end{lemma}

\begin{proof}
  We define the subspace
  $U = \linspan\{u_1, \dotsc, u_{N-1}, v_1, \dotsc, v_{N-1} \}$ and
  let $d = \dim (U) \le 2 (N-1)$.
  Let $R : U \to U$ be a linear operator such that
  $Ru_i = v_i$ for $i = 1, 2, \dots, N-1$.  Possibly after
  permutation of the family $\{v_n\}$ , we can choose $R$ such that
  $R \in \SO(d)$ (there may be different choices of $R$).
  Now, since $SO(d)$ is pathwise-connected
  (see e.g.~\cite[Section~10.5]{Se}), there exists a continuous path
  $\tilde{\gamma} : [0,1] \to \SO(d)$ such that
  $\tilde{\gamma}(0) = \Id$ and $\tilde{\gamma}(1) = R$.  Let us
  define the map,
  \begin{equation*}
    \gamma \colon [0,1]\times \mathbb{S}^{N-2} \to \M:
    (t,a_1,\dots,a_{N-1})
    \mapsto \sum_{i=1}^{N-1} a_i\,\tilde{\gamma}(t)(u_i).
  \end{equation*}
  It is clear that $\gamma$ is continuous, $\gamma(t, \cdot)$ is odd
  for all $t$, and $\gamma(0, \cdot) = \gamma_0$,
  $\gamma(1, \cdot) = \gamma_1$.
\end{proof}

We now turn to the second hypothesis, which requires finding conditions to ensure that the strict inequality \eqref{SMP2} in Theorem \ref{Objective1} is satisfied. At this point, we shall rely on some results from \cite{BaLiLi}. In particular, the next Lemma is essentially~\cite[Lemma~3.2]{BaLiLi}.

\begin{lemma}\label{lemma:intersection}
  Let $L_1, L$ be finite dimensional normed vector spaces such that
  $\dim(L_1) < \dim(L).$ Let $S = \{u\in L \mid \|u\| = 1\}$,
  $\alpha \in \R$ and
  $H=(H_1,H_2)\colon [0,1]\times S \to \R \times L_1$ be a
  continuous map such that, for all $t$,
  $u \mapsto H_1(t,u)$ is even, $u \mapsto H_2(t,u)$ is
  odd, and
  \begin{equation*}
    H_1(0,u) < \alpha < H_1(1,u), \quad \text{for } u \in S.
  \end{equation*}
  Then there exists $(t,u)\in [0,1]\times S$ such that
  $H(t,u)=(\alpha,0)$.
\end{lemma}

In the proof of our next result we are inspired by \cite[Lemma~3.3]{BaLiLi},
see also \cite[Lemma~2.3]{BaVa}.

\begin{theorem}\label{prop:MPcondition}
Let $\Phi : E \to \R$ be a continuous even functional,
    $d \in \N$, and $\gamma_i : \mathbb{S}^d \to S_\mu$, $i = 1,2$,
    be two odd functions.  Assume that the set
    \begin{equation*}
      \Gamma
      := \bigl\{ \gamma \in C([0,1] \times \mathbb{S}^d, S_\mu) \bigm|
      \forall t \in [0,1],\ \gamma(t,\cdot) \text{ is odd, }
      \gamma(0,\cdot) = \gamma_0, \text{ and }
      \gamma(1,\cdot) = \gamma_1 \bigr\}
    \end{equation*}
    is not empty.  Assume further that there exists a
    continuous even functional
    $J : E \to \R$, $\beta \in \R$, and $W \subset E$ a subspace with
    $\dim W \le d$ such that
  
  \begin{itemize}
  \item[(H1)] $J(\gamma_0(s)) < \beta < J(\gamma_1(s))$, for all
    $s \in \mathbb{S}^d$;
  \item[(H2)]
$\displaystyle \max_{s\in \mathbb{S}^d}
    \max\bigl\{\Phi(\gamma_0(s)), \Phi(\gamma_1(s))\bigr\}
    < \inf_{u \in B} \Phi(u)$ \quad where \quad $B := \{ u \in S_\mu \cap W^\perp \mid J(u) = \beta \}$.
    
  \end{itemize}
  Then
  \begin{equation}\label{SMP2L}
    c := \inf_{\gamma\in\Gamma} \, \max_{(t, s) \in [0,1] \times \mathbb{S}^d}
    \Phi(\gamma(t,s))
    \ge \inf_{u \in B} \Phi(u).
  \end{equation}
\end{theorem}

\begin{proof}
  Let $\gamma \in \Gamma$ be arbitrary and $P\colon E\to W$ be
  the orthogonal projection. Define
  \begin{equation*}
    h\colon \M \to \R \times W : u \mapsto (J(u), Pu)
    \quad\text{and}\quad
    H = h\circ \gamma\colon [0,1]\times \mathbb{S}^d \to \R\times W.
  \end{equation*}
  Setting $L = \R^{d+1}$, $S = \mathbb{S}^{d}$, $L_1= W$ and
  $\alpha = \beta$, we see that $L$, $S$, $L_1$, $\alpha$
  and $H$ satisfy all
  the conditions of Lemma \ref{lemma:intersection}.
  Therefore there exists
  $(t_0,s_0) \in [0,1] \times \mathbb{S}^{d}$ such that
  $H(t_0,s_0) = (\beta, 0)$. That is $\gamma(t_0,s_0) \in B$.
  One deduces that
  \begin{equation*}
    \max_{(t,s)\in [0,1]\times \mathbb{S}^d} \Phi(\gamma(t,s))
    \ge \Phi(\gamma(t_0,s_0))
    \ge \inf_{u \in B} \Phi(u) .
  \end{equation*}
  This proves that \eqref{SMP2L} holds since $\gamma \in \Gamma$ is
  arbitrary.
\end{proof}

\section{A Poho\v{z}aev type identity and its consequences}
\label{Preliminaries_II}

In this section we focus on deriving properties of solutions
to~\eqref{1.2LL} when $\lambda=0$. Observe that since we are assuming that the
compact core is non trivial: $\G$ has at least one bounded edge and
thus there is at least one edge where the nonlinearity is acting.
Some considerations in this section are slightly more general than
what is needed to prove Theorem \ref{thm: Whole ex}.

\smallskip

First let us recall that if $u$ is solution to 
$- u'' + \lambda u = \rho|u|^{p-2}u$
on some interval $I \subseteq \R$, then the function
\begin{equation*}
  H_u(x) := \frac{1}{2} \bigl(u'(x)\bigr)^2 + V_{\lambda}(u(x))
  \quad \text{where} \quad
  V_{\lambda}(u) := \frac{\rho}{p} |u|^p - \frac{\lambda}{2} |u|^2
\end{equation*}
is constant on $I$. Indeed,
$H_u'(x) := u'(x) \cdot (u''(x) + V_{\lambda}'(u(x))) = 0$.
We call this constant $H_u$ the \emph{ODE energy}
of the solution $u$ on $I$.

\begin{proposition}[Poho\v{z}aev identity on metric graphs]
  \label{thm:pohozaev}
  Let $\G$ be a metric graph with finitely many edges (bounded or
  not).  Let $p > 2$, $\lambda \in \R$, and $u \in H^1(\G)$ be a
  solution to \eqref{1.2LL}.  For each bounded edge $e$ of $\G$, let
  the {\it ODE energy} of the solution $u$ on $e$ be given by
  \begin{equation}\label{DefH}
    H_u(e) := H_u(x)
    = \frac{1}{2} |u'(x)|^2
    + \frac{\rho}{p} |u(x)|^p - \frac{\lambda}{2} |u(x)|^2,
  \end{equation}
  where $x$ is an arbitrary point of $e$.  Finally, define
  \begin{equation}
    \label{eq:defP}
    P_{\rho}(u,\G) := \sum_{e \text{ is a bounded edge of $\G$}} \ell_e\, H_u(e)
  \end{equation}
  where $\ell_e$ is the length of the edge $e$. 
  Then, one has
  \begin{equation*}
    \frac{1}{2} \| u' \|_{L^2(\G)}^2
    + \frac{\rho}{p} \|\kappa u \|_{L^p(\G)}^p
    = \frac{\lambda}{2} \| u \|_{L^2(\G)}^2 + P_{\rho}(u,\G).
  \end{equation*}
\end{proposition}

\begin{proof}
  Let $e$ be a bounded edge of $\G$.  We identify it with the
  interval $\intervalcc{0, \ell_e}$.
  Integrating \eqref{DefH} on $e$, we get
  \begin{equation}
    \label{pohozaev_e}
    \frac{1}{2} \| u' \|_{L^2(e)}^2
    + \frac{\rho}{p} \| \kappa u \|_{L^p(e)}^p
    = \frac{\lambda}{2} \| u \|_{L^2(e)}^2 + \ell_e \, H_u(e).
  \end{equation}
  Note that \eqref{pohozaev_e} also holds when $e$ is a half-line
  if in this
  case we set $\ell_e \, H_u(e) := 0$  since $\kappa|_e = 0$ and $u \in H^1(\G)$.
  We end the proof by taking
  the sum of \eqref{pohozaev_e} over all edges of $\G$ (whether
  bounded or not).
\end{proof}

\begin{lemma}
  \label{lemma:link}
  Let $\G$ be a metric graph with finitely many edges (bounded or
  not).  Let $p > 2$ and $\lambda \in \R$. Let $u \in H^1(\G)$ be a
  solution to \eqref{1.2LL}.  Then, one has
  \begin{equation*}
    E_{\rho}(u, \G)
    = \frac{(p-6) \lambda}{2(p+2)} \| u \|_{L^2(\G)}^2
    + \frac{p-2}{p+2} P_{\rho}(u,\G),
  \end{equation*}
  where $P_{\rho}(u,\G)$ is defined by \eqref{DefH}--\eqref{eq:defP}.
\end{lemma}

\begin{proof}
  First note that multiplying
  $-u'' + \lambda u = \rho \kappa(x) |u|^{p-2}u$ by $u$ and
  integrating over $\G$ (taking into account the Kirchhoff boundary
  conditions) we get
  \begin{equation}\label{Nehari}
    \| u' \|_{L^2(\G)}^2
    + \lambda \| u \|_{L^2(\G)}^2
    = \rho \| \kappa u \|_{L^p(\G)}^p.
  \end{equation}
  From Proposition~\ref{thm:pohozaev} and \eqref{Nehari}, we obtain
  \begin{align*}
    \| u' \|_{L^2(\G)}^2
    &= \frac{(p-2) \lambda}{p+2} \| u \|_{L^2(\G)}^2
      + \frac{2p}{p+2} P_{\rho}(u,\G),\\
    \rho \| \kappa u \|_{L^p(\G)}^p
    &= \frac{2p \lambda}{p+2} \| u \|_{L^2(\G)}^2
      + \frac{2p}{p+2} P_{\rho}(u,\G).
  \end{align*}
  Thus
  \begin{equation*}
    E_{\rho}(u, \G)
    = \frac{1}{2} \| u' \|_{L^2(\G)}^2
    - \frac{\rho}{p} \| \kappa u \|_{L^p(\G)}^p
    = \frac{(p-6) \lambda}{2(p+2)} \| u \|_{L^2(\G)}^2
    + \frac{p-2}{p+2} P_{\rho}(u,\G).
    \qedhere
  \end{equation*}
\end{proof}


Let us now establish some relationships between the ODE energy $H_u$
of a solution $u$ on an interval and its $L^2$-norm in the case
$\lambda = 0$. 

\begin{lemma}
  \label{bound_L2_norm}
  Let $\alpha > 0$ and $p \ge 2$.  Let
  $u: \R \to \R$ be an $\tau$-periodic\footnote{I.e. so that
    $u(x + \tau) = u(x)$ for all $x \in \R$.
    It is not necessary for $\tau$ to be the minimum period.} solution of
  \begin{equation*}
    -u'' = \alpha |u|^{p-2} u.
  \end{equation*}
  for some $\tau > 0$.  Let $H_u$ be the \emph{ODE energy} of the
  solution $u$.
  Then,
  \begin{equation}
    \label{estimates_L_infinity_norm}
    \frac{\tau}{8} \Bigl( \frac{p H_u}{\alpha} \Bigr)^{2/p}
    = \frac{\tau}{8} \|u\|_{L^\infty}^2
    \le \int_0^{\tau} |u(x)|^2 \intd x
    \le \tau \|u\|_{L^\infty}^2
    = \tau \Bigl( \frac{p H_u}{\alpha} \Bigr)^{2/p}.
  \end{equation}
\end{lemma}

\begin{proof}
  It is advantageous to consider that we are studying periodic
  solutions of the equation of motion in the potential well defined by
  $$V(u) := \frac{\alpha |u|^p}{p},$$
  since the equation reads $u'' = -V'(u)$.  The \emph{ODE energy} of
  the solution $u$, which is given here by
  \begin{equation*}
    H_u := \tfrac{1}{2} (u')^2 + V(u), 
  \end{equation*}
  is constant with respect to time.  We immediately obtain that
  \begin{equation*}
    \tfrac{1}{2} (u'(x))^2
    \begin{cases}
      \le H_u - V(\|u\|_{L^\infty}/2)
      &\text{for all } x \in \intervalcc{0, \tau}
        \text{ such that } |u(x)| \ge \|u\|_{L^\infty}/2,\\
      \ge H_u - V(\|u\|_{L^\infty}/2)
      &\text{for all } x \in \intervalcc{0, \tau}
        \text{ such that } |u(x)| \le \|u\|_{L^\infty}/2.
    \end{cases}
  \end{equation*}
  Therefore, a particle in the potential well always has a smaller
  speed (in absolute value) when going through the region
  $\intervalcc{-|u|_{\infty}, -|u|_{\infty}/2} \cup
  \intervalcc{|u|_{\infty}/2, |u|_{\infty}}$ than when going through
  the region $\intervalcc{-|u|_{\infty}/2, |u|_{\infty}/2}$.  Since
  both those regions have the same length, we deduce that
  \begin{equation*}
    |A| \ge \tfrac{1}{2}\tau
    \qquad \text{where }
    A := \bigl\{ x \in \intervalcc{0, \tau}
    \bigm| |u(x)| \ge \tfrac{1}{2} \|u\|_{L^\infty}
    \bigr\}
  \end{equation*}
  as the particle will spend at least half its time in the region
  where it has a slower speed.
		
  Regarding the inequalities in \eqref{estimates_L_infinity_norm},
  the upper bound is trivial, and the lower bound follows from the
  inequalities
  \begin{equation*}
    \int_0^{\tau} |u(x)|^2 \intd x
    \ge \int_{A} |u(x)|^2 \intd x
    \ge \tfrac{1}{4} \|u\|_{L^\infty}^2 \, |A|
    \ge \tfrac{1}{8} \tau\, \|u\|_{L^\infty}^2.
  \end{equation*}
  Finally, the equalities in \eqref{estimates_L_infinity_norm}
  follow from the fact that,
  for periodic solutions, one has
  \begin{equation*}
    H_u = V(\|u\|_{L^\infty}) = \frac{\alpha \|u\|_{L^\infty}^p}{p},
  \end{equation*}
  since the derivative of the solution vanishes at any maximum or
  minimum point.
\end{proof}

\begin{lemma}\label{lem_determined}
  Let $\alpha > 0$ and $p \ge 2$.  The solution of equation
  \begin{equation*}
    -u'' = \alpha |u|^{p-2} u,
  \end{equation*}
  with initial conditions $u'(0) = 0$ and $u(0) = u_0 > 0$ is
  $\tau(u_0)$-periodic, where
  \begin{equation*}
    \tau(u_0) := \frac{C(p)}{\sqrt{\alpha}} u_0^{(2-p)/{2}}
  \end{equation*}
  for some constant $C(p) > 0$.  Its ODE energy is given by
  $H_u = V(u_0) = \frac{\alpha}{p} u_0^p$.  Moreover, it is (up to
  time translations) the unique solution of the ODE with this energy,
  and the unique solution of the ODE with this period.
\end{lemma}

\begin{proof}
  It is a standard fact (see e.g. \cite[p.~18]{Ar}) that the period is
  given by
  \begin{equation*}
    \tau(u_0)
    = 2 \int_{-u_0}^{u_0} \frac{\intd u}{\sqrt{2(V(u_0) - V(u))}}
    = \sqrt{\frac{8p}{\alpha}}
    \int_{0}^{u_0} \frac{\intd u}{\sqrt{u_0^p - u^p}}
    = \Biggl( \sqrt{\frac{8p}{\alpha}}
    \int_{0}^{1} \frac{\intd t}{\sqrt{1 - t^p}} \Biggr)
    u_0^{1-p/2}.
  \end{equation*}
  The claim about the energy follows from the definitions.   
  The fact that the set \begin{equation*}
    \bigl\{ (u, v) \in \R^2 \bigm| \tfrac{1}{2}v^2 + V(u) = h \bigr\}
  \end{equation*}
  is empty for $h < 0$, is $\{ (0, 0) \}$ for $h = 0$ and is a simple
  closed curve for $h > 0$ implies that no other solutions have this same energy since by a phase plane analysis we obtain that there is a unique orbit of energy $h$ for every $h>0$. We then deduce from the previous computations
  that this orbit corresponds to a
  solution of period $C(\alpha, p) u_0^{1-p/2}$, which ends
  the proof as the map
  \begin{equation*}
    \intervaloo{0, +\infty} \to \intervaloo{0, +\infty}:
    u_0 \mapsto C(\alpha, p) u_0^{1-p/2}
  \end{equation*}
  is decreasing, so all orbits correspond to solutions
  with different periods.
\end{proof}

From here we may deduce that functions with a high ODE energy necessarily have a high $L^2$ norm.

\begin{corollary}\label{below_estimate}
  Let $\ell > 0$,
  $0 < \underline{\alpha} < \overline{\alpha} < \infty$,
  and $p > 2$.
  For every $\underline{\mu} > 0$, there exists $\underline{H} > 0$
  such that if $u: \intervalcc{0, \ell} \rightarrow \R$ is a solution to
  \begin{equation*}
    -u'' = \alpha |u|^{p-2} u,
    \qquad \text{with }
    \alpha \in \intervalcc{\underline{\alpha}, \overline{\alpha}}
    \text{ and }
    H_u \ge \underline{H},
  \end{equation*}
  then 
  \begin{equation*}
    \int_0^{\ell} |u(x)|^2 \intd x \ge \underline{\mu}.
  \end{equation*}
\end{corollary}

\begin{proof}
  Lemma \ref{lem_determined} implies that if the ODE
    energy $H_u \ge \underline{H}$, then
    $u_0 \ge (p \underline{H} / \overline{\alpha} )^{1/p}$ and so
    \begin{equation*}
      \tau(u_0)
      \le \frac{C(p)}{\sqrt{\underline{\alpha}}}
      \biggl(\frac{p\, \underline{H}}{\overline{\alpha}}
      \biggr)^{(2-p)/(2p)} .
    \end{equation*}
  Thus, if $\underline{H}$ is large
  enough, $u$ is periodic with a period $\tau$ less then
  $\ell /2$. There thus exists some interval
  $\intervalcc{0,k\tau} \subseteq \intervalcc{0, \ell}$,
  with $k \in \N$ and $k\tau \ge \ell/2$.
  Thus $u$ is $k\tau$-periodic
  and its ODE energy is at least $\underline{H}$.
  Lemma~\ref{bound_L2_norm} implies that the $L^2$-norm of $u$
  on $\intervalcc{0,k\tau}$ can be made arbitrarily high
 taking $\underline{H}$
  large enough.
  This ends the proof.
\end{proof}
	
The last result of this section will be crucially used to rule out the
possibility that the Lagrange multiplier associated to a weak limit of
some Palais-Smale sequence is $0$.

\begin{proposition}\label{forbid}
  Let $\G$ be a metric graph with finitely many edges (bounded or not)
  and $p > 2$. Let $\{u_n\} \subset H^1(\G)$
  and $\{\rho_n\} \subset \intervalcc{1/2, 1}$
    be sequences such that $u_n$ is a solution
  to \eqref{1.2LL}
  with $\rho = \rho_n$ and
  $\lambda =0$.
  If
  $E_{\rho_n}(u_n, \G) \to + \infty$, then
  $\|u_n\|_{L^2(\G)} \to \infty$.
\end{proposition}

\begin{proof}
  First notice that it is sufficient to prove that, up to a
    subsequence, $\|u_n\|_{L^2(\G)} \to \infty$ because, replaying the
    argument on an arbitrary subsequence of $\{u_n\}$ will give a
    sub-subsequence which converges to infinity, which is equivalent
    to the claim.

  Since $E_{\rho_n}(u_n, \G) \to + \infty$ and $\lambda = 0$,
  Lemma~\ref{lemma:link} implies that $P_{\rho_n}(u_n, \G) \to +\infty$.
  Let $e_0 \in \mathcal{E}$ be a bounded edge such that
  $\ell_{e_0} H_{u_n}(e_0) \ge \ell_{e} H_{u_n}(e)$ for all bounded
  edges $e \in \mathcal{E}$.  Given that $\mathcal{E}$ is finite, it
  is possible to select $e_0$ independent of $n$, taking subsequences
  of $\{u_n\}$ and $\{\rho_n\}$ if necessary.  Since
  \begin{equation*}
    P_{\rho_n}(u_n, \G)
    \le \card(\mathcal{E}) \, \ell_{e_0} H_{u_n}(e_0),
  \end{equation*}
  $ H_{u_n}(e_0) \to + \infty$.  At this point, using
  Corollary~\ref{below_estimate} with $u = u_n$, $\alpha =\rho_n$, and
  $[0, \ell] = e_0$, we deduce that
  $\|u_n\|_{L^2(e_0)} \to \infty$ and thus
  $\|u_n\|_{L^2(\G)} \to \infty$.
\end{proof}

\section{Infinitely many minimax levels for $E_{\rho}$
  for almost every $\rho \in [\frac{1}{2}, 1]$}
\label{Proof_Th2}

This aim of this section is to prove the following result.
\begin{proposition}\label{Prop: inftylevels}
  For any $\mu > 0$ and $p > 2$, there exists $N_0 \in \N$ so that if $N \ge N_0$,
  there exist functions $\gamma_{0,N}$ and $\gamma_{1,N}$ such
  that the family of functionals
  \begin{equation*}
    E_\rho(\cdot,\G) : H^1(\G) \to \R :
    u \mapsto
    \frac{1}{2}\int_\G|u'|^2-\frac{{\rho}}{p}\int_\K|u|^p,
    \qquad \rho \in \left[\frac{1}{2},1\right]
  \end{equation*}
  satisfies the assumptions of Theorem \ref{Objective1}. In
  particular,
  \begin{equation}\label{SMP1LL}
    \Gamma_N
    = \bigl\{ \gamma\in C([0,1]\times \mathbb{S}^{N-2}, H_{\mu}^1(\G))
    \bigm| \forall t \in [0,1],\ \gamma(t,\cdot) \text{ is odd, }
    \gamma (0, \cdot) = \gamma_{0,N}, \text{ and }
    \gamma(1, \cdot) = \gamma_{1,N} \bigr\}
  \end{equation}
  is non void and 
  \begin{equation}\label{SMP2LL}
    c_{\rho}^N
    = \inf_{\gamma\in\Gamma_N}\max_{(t,s) \in [0,1] \times \mathbb{S}^{N-2}}
    E_\rho(\gamma(t,s), \G))
    > \max_{s \in \mathbb{S}^{N-2}}
    \max\bigl\{E_\rho(\gamma_0(s), \G)), E_\rho (\gamma_1(s), \G)\bigr\},
    \quad \rho \in \left[\frac{1}{2},1\right]
  \end{equation}
  Furthermore, $c_{\rho}^N \xrightarrow[N\to+\infty]{} +\infty \, $
  uniformly w.r.t.\ $\rho \in \intervalcc{1/2, 1}$.
    In particular there are infinitely many distinct values of
    $c_{\rho}^N$.
\end{proposition}

\begin{remark}\label{c_rho_N_finite}
    Note that the levels $c_{\rho}^N$ are real numbers for every $N \ge N_0$
    and every $\rho \in \intervalcc{\frac{1}{2}, 1}$
    since they are defined by infima over nonempty sets (thus $c_{\rho}^N < +\infty$) and that
    inequality \eqref{SMP2LL} implies that $c_{\rho}^N > -\infty$.
\end{remark}

We consider Theorem \ref{Objective1} with the choice of the family $\Phi_{\rho}= E_\rho(\cdot, \G)$. Also $E=H^1(\G)$, $H=L^2(\G)$, $S_\mu = H^1_\mu(\G)$, 
Setting
\begin{equation*}
  A(u) = \frac12\int_{\G} |u'|^2\intd x
  \quad \text{and} \quad
  B(u) = \frac{\rho}{p}\int_{\K} |u|^p\intd x,
\end{equation*}
assumption \eqref{hp coer} holds, since we have that
\[
u \in H^1_\mu(\G), \ \|u\|_{H^1(\G)} \to +\infty \quad \implies \quad A(u) \to +\infty.
\]
Let $E'_{\rho}$ and $E''_{\rho}$ denote respectively the free first
and second Fr\'echet derivatives of $E_{\rho}$.
Note that $B''$, whence $E''_\rho$, is
$\min\{p-2, 1\}$-H\"older continuous
on bounded sets of~$H^1(\G)$, which,
in view of Remark~\ref{rem:Holder-cont},
implies that assumption \eqref{Holder} holds.
As such, it only remains to show that the two hypothesis posed on
$\Gamma_N$ hold. This is where Lemma \ref{Pierotti-Verzini} and
Theorem \ref{prop:MPcondition}  will come into play.

\medskip

The following two lemmas will provide orthogonal families
to be used in Lemma~\ref{Pierotti-Verzini}.

\begin{lemma}\label{lemma:eigennotfunctions}
  Let $\G$ be a graph satisfying \eqref{HG}, $p > 2$ and $\mu >0$.
  For any $\beta > 0$, there exists a sequence of functions
  $\{\varphi_1,\varphi_2,\dots\}$ such that for any $i, j\in \N^*$
  and any $\rho \in [ \frac{1}{2}, 1]$:
  \begin{itemize}
  \item[(i)] $\varphi_i\in S_\mu$; \quad
    $\|\varphi_i'\|_{L^2(\G)}=\beta$; \quad
    $E_\rho(\varphi_i, \G)=\beta^2/2$;
  \item[(ii)] $\varphi_i$ has compact support and
    $\supp(\varphi_i)\cap \supp(\varphi_j)=\emptyset$ for $i\neq j$;
  \item[(iii)] for any $N \ge 2$ and $a\in \mathbb{S}^{N-2}$,
    \begin{math}
      \bigl\| \bigl(\sum_{i=1}^{N-1}a_i \varphi_i \bigr)'\bigr\|_{L^2(\G)}
      = \beta
    \end{math}
    and
    \begin{math}
      E_\rho\bigl(\sum_{i=1}^{N-1}a_i\varphi_i, \, \G\bigr)
      = \beta^2/2  
    \end{math}%
    .
  \end{itemize}
\end{lemma}
\begin{proof}
  Let $\varphi\in C_{\text{c}}^{\infty}(\R)$ be a function supported
  on the interval $\intervaloo{0,1}$
  such that $\|\varphi\|_{L^2(\R)}^2=\mu$.  Define,
  for $t\in \R^+$ the function $\varphi^t$ by
  \begin{equation}
    \label{Eq:dilation}
    \varphi^{t}(x) := t^{1/2}\varphi(tx).
  \end{equation}
  If we now view $\varphi$ as a function in $H^1(\G)$ whose support
  is contained in a half-line which we identify with
  $\intervalco{0, \infty}$, we can define
  \begin{equation*}
    \varphi_1 := \varphi^\tau
    \quad \text{with} \quad
    \tau := \frac{\beta}{\|\varphi'\|_{L^2(\G)}}.
  \end{equation*}
  The function $\varphi_1$ satisfies (i). Indeed,
  for any $t>0$, $\|\varphi^t\|_{L^2(\G)} = \|\varphi\|_{L^2(\G)}=\mu$, and a
  direct calculation yields
  \begin{equation}
    \label{Eq:dilationderivative}
    \|\varphi_1'\|_{L^2(\G)}^2
    = \tau^2 \|\varphi'\|_{L^2(\G)}^2
    = \beta^2.
  \end{equation}
  Finally, since $\varphi_1$ is supported in the half-line, we have
  \begin{equation*}
    E_\rho(\varphi_1,\G)
    = \tfrac{1}{2} \|\varphi_1'\|_{L^2(\G)}^2
    = \frac{\beta^2}{2}.    
  \end{equation*}
  Define now, for $i \geq 2$,
  \begin{equation*}
    \varphi_i(x)
    := \varphi_1\left(x- \frac{i-1}{\tau} \right).
  \end{equation*}
  Since the $\varphi_i$ are translations of $\varphi_1$ they still
  satisfy (i).  Also, observe that by definition
  $\supp(\varphi_i) \subset \bigintervaloo{\frac{i-1}{\tau},
    \frac{i}{\tau}}$ and so they all have disjoint compact
  supports.  This
  is (ii). Finally, observe that for $a\in \mathbb{S}^{N-2}$
  \begin{equation*}
    \left\|\biggl(\sum_{i=1}^{N-1}a_i\varphi_i\biggr)^{'} \right\|_{L^2(\G)}^2
    = \sum_{i=1}^{N-1}a_i^2 \, \|\varphi_i'\|_{L^2(\G)}^2
    = \beta^2,
  \end{equation*}
  from which (iii) follows, ending the proof. 
\end{proof}

\begin{lemma}\label{lemma:eigendilationfunctions}
  Let $\G$ be a graph satisfying \eqref{HG}, $p>6$ and $\mu >0$. For
  any fixed integer $N \ge 2$  and any given values of
  $\overline{\beta} > 0$, $\overline{b} > 0$, there exist functions
  $\overline{\varphi}_1,\dots,\overline{\varphi}_N$, compactly
  supported in $\K$, such that for all $i,j\in \{1,\dots, N\}$ and all
  $\rho \in [1/2, 1]$,
  \begin{itemize}
  \item[(i)] $\overline{\varphi}_i \in S_\mu$; \quad
    $\|\overline{\varphi}_i'\|_{L^2(\G)}\geq\overline{\beta}$; 
  \item[(ii)] $\supp(\overline{\varphi}_i)
    \cap \supp(\overline{\varphi}_j) = \emptyset$ for $i\neq j$;
  \item[(iii)] if $a\in \mathbb{S}^{N-2}$ then
    \begin{math}
      \bigl\| \bigl(\sum_{i=1}^{N-1}a_i \overline{\varphi}_i
      \bigr)' \bigr\|_{L^2(\G)}
      \ge \overline{\beta}
    \end{math}
    and
    \begin{math}
      E_\rho\bigl(\sum_{i=1}^{N-1}a_i\overline{\varphi}_i, \, \G\big)
      \le \overline{b}
    \end{math}%
    .
  \end{itemize}
\end{lemma}
\begin{proof}
  Let $\edge=[0,\ell_\edge]$ be any bounded edge of $\G$. Let
  $\varphi\in C^{\infty}_c
  \bigl(\intervaloo{0, \ell_\edge/N} \bigr)$
  be any function such that
  $\|\varphi\|_{L^2(\R)}=\mu$.  Using the notation 
  \eqref{Eq:dilation}, we notice that
  \begin{math}
    \supp(\varphi^{t})\subset \intervaloo{0, \ell_\edge/N}
  \end{math}
  whenever $t\geq 1$.
  Define the functions
  \begin{equation*}
\overline{\varphi}_i
    := \varphi^{t} \left(x-\frac{(i-1)\ell_\edge}{N}\right),
    \qquad
    i = 1,\dots, N,
  \end{equation*}
  where $t \ge 1$ will be chosen later.
  Note that
  \begin{equation*}
    \supp(\overline{\varphi}_i)
    \subset
    \left(\frac{(i-1)\ell_\edge}{N}, \frac{i\ell_e}{N}\right)
  \end{equation*}
  so the functions $\overline{\varphi}_i$ have disjoint supports and
  (ii) is satisfied.
  Viewing now
   $\overline{\varphi}_{i}$ as functions in $H^1(\G)$
  supported in $\edge$, we may compute the
  energy of the function $\sum_{i=1}^{N-1}a_i \overline{\varphi}_{i}$
  with $a \in \mathbb{S}^{N-2}$
  as follows
  \begin{align*}
    E_\rho\biggl(\sum_{i=1}^{N-1}a_i \overline{\varphi}_{i}, \, \G\biggr)
    &= \frac{1}{2} \int_\edge {\textstyle \Bigl| \sum\limits_{i=1}^N
      a_i \overline{\varphi}_{i}'\Big|^2} \intd x
      - \frac{\rho}{p} \int_\edge \Bigr| {\textstyle \sum\limits_{i=1}^{N-1}
      a_i \overline{\varphi}_{i} \Bigr|^p \intd x }\\
    &= \frac{t^2}{2}\sum_{i=1}^{N-1} a_i^2
      \int_0^{\ell_\edge/N} |\varphi'|^2 \intd x
      -\frac{\rho \, t^{(p-2)/2}}{p} \,\sum_{i=1}^{N-1}|a_i|^p
      \int_0^{\ell_\edge/N} |\varphi|^p \nonumber \\
    & \le 
      \frac{t^2\|\varphi'\|_{L^2(\G)}^2}{2}
      - \frac{Ct^{(p-2)/2} \, \|\varphi\|_{L^p(\K)}^p}{2p}
      \xrightarrow[t\to+\infty]{} -\infty
  \end{align*}
  where
  $C := \min_{a\in\mathbb{S}^{N-2}} \sum_{i=1}^{N-1}|a_i|^p$.  Thus,
  for all $\overline{b}\in\R$, there exists $T_0>0$ such that for all
  $t>T_0$ we have
  $E_\rho\big(\sum_{i=1}^Na_i \overline{\varphi}_{i}, \G\big)<\overline{b}$.  As
  a result,
  if we choose
  \begin{equation*}
    t
    :=\max\left\{1,
      \frac{\overline{\beta}}{\|\varphi'\|_{L^2(\G)}},
      T_0\right\},
  \end{equation*}
  the functions $\overline{\varphi}_i$ satisfy all of the desired
  properties.  Indeed, $\overline{\varphi}_i \in H^1_{\mu}(\G)$ and
  from \eqref{Eq:dilationderivative} we have
  \begin{equation}
    \label{eq:lower-bound-H10}
    \|\overline{\varphi}_i'\|_{L^2(\G)}
    = t \|\varphi'\|_{L^2(\G)}
    \ge \overline{\beta},    
  \end{equation}
  which implies (i).
  Finally, the choice of $t$, \eqref{eq:lower-bound-H10} and
  \begin{math}
    \bigl\| \bigl(\sum_{i=1}^{N-1}a_i \overline{\varphi}_i
    \bigr)' \bigr\|_{L^2(\G)}^2
    = \sum_{i=1}^{N-1}a_i^2 \|\overline{\varphi}_i'\|_{L^2(\G)}^2
  \end{math}
  show
  (iii), ending the proof.
\end{proof}

Now let $\{V_N\}$ be a sequence of linear subspaces of
$H^1(\G)$ with $ \dim(V_N) = N$ which is exhausting $H^1(\G)$ in the
sense that $$\bigcup_{N \geq 1}V_N$$ is dense in $H^1(\G)$. We recall
that for separable Hilbert spaces, such as $H^1(\G)$,
such a sequence always exists. \medskip

Our next lemma is an adaptation of~\cite[Lemma~2.1]{BaVa}.
\begin{lemma}\label{lemma: mu_n divergence}
  For any $p >2$ there holds:
  \begin{equation*}
    S_N
    :=
    \inf_{u\in V_{N-2}^\perp}\frac{\int_\G|u'|^2 + |u|^2}{
      \left(\int_\K|u|^p\right)^{2/p}}
    \to \infty, \quad \text{as } \,  N \to \infty.
  \end{equation*}
\end{lemma}
\begin{proof}
  Suppose by contradiction that there exists a sequence
  $\{u_N\} \subset V_{N-2}^\perp$ such that $\|u_N\|_{L^p(\K)}=1$ and
  $\|u_N\|_{H^1(\G)}$ is bounded.  In particular,
  up to a subsequence, there exists $u \in H^1(\G)$ such that
  $u_N\rightharpoonup u$ in ${H^1(\G)}$ (and thus in $H^1(\K)$) and
  therefore $u_N\to u$ in $L^p(\K)$.  Let $v\in H^1(\G)$.
  Because $\{V_N\}$ exhausts $H^1(\G)$, there exists a sequence
  $\{v_N\} \subset H^1_{\mu}(\G)$ such that, for all $N \in \N$,
  $v_N\in V_{N-2}$ and $v_N\to v$ in $H^1(\G)$. Taking the
  scalar product in ${H^1(\G)}$ we have
  \begin{equation*}
    |\langle u_N,v\rangle|
    \le |\langle u_N,v-v_N\rangle| + |\langle u_N,v_N\rangle|
    = |\langle u_N,v-v_N\rangle|
    \le \|u_N\|_{H^1(\G)} \|v-v_N\|_{H^1(\G)}
    \xrightarrow[N\to\infty]{} 0.
  \end{equation*}
  It follows that $u_N\rightharpoonup 0=u$ in contradiction with
  $\|u_N\|_{L^p(\K)}=1$.
\end{proof}

\noindent We now define 
\begin{equation}\label{Eq: beta_n}
  \beta_N:=\left(\frac{S_N^{p/2}}{L}\right)^{1/(p-2)}
  \quad \text{where} \quad
  L = L(p) := \frac{3}{p}\max_{x>0}
  \frac{(\mu+x^2)^{p/2}}{\mu+x^p}.
\end{equation}
As an immediate consequence of Lemma \ref{lemma: mu_n divergence}, we have that $\beta_N\to \infty$. Thus if we define 
\begin{equation}\label{Eq:b_n}
  b_{\rho}^N := \inf_{u\in B_N} E_\rho(u, \G)
  \quad \text{where} \quad
  B_N := \bigl\{u\in V_{N-2}^\perp\cap H_{\mu}^1(\G) \bigm|
  \|u'\|_{L^2(\G)} = \beta_N \bigr\}
\end{equation}
we obtain that
\begin{lemma}\label{lemma:b_n divergence}
  $b_{\rho}^N\to + \infty$ as $N \to + \infty$, uniformly in
  $\rho\in [1/2,1]$.
\end{lemma}
\begin{proof}
  For every $u\in B_N$ we have 
  \begin{equation*}
    \begin{split}
      E_\rho(u, \G)
      &=\frac{1}{2}\int_\G |u'|^2
        - \frac{\rho}{p}\left(\int_\K |u|^p\right)^{\frac{2}{p}\cdot\frac{p}{2}}
        \geq \frac{1}{2}\int_\G |u'|^2
        -\frac{1}{p}\left(\frac{\mu+\int_\G|u'|^2}{S_N}
        \right)^{p/2}\\
      &\ge\frac{1}{2} \|u'\|_{L^2(\G)}^2
        - \frac{L}{3S_N^{p/2}} \left(\mu + \|u'\|_{L^2(\G)}^{p}\right)\\
      &= \frac{1}{2}\beta_{N}^2-\frac{1}{3}\beta_N^{2-p}(\mu+\beta_N^p)\\
      &
        =\frac{1}{6}\beta_N^2 + o(1).
    \end{split}
  \end{equation*}
  The proof is completed by taking the infimum over $B_N$.
\end{proof}

We are finally in position to give the

\begin{proof}[Proof of Proposition \ref{Prop: inftylevels}]
  We have already proved that \eqref{Holder} and \eqref{hp coer} hold.
  Let $\{V_N\}$ with $\dim(V_N) = N$ be an exhausting
  sequence of $H^1(\G)$ and, for each $N \geq 2$,
  define the values $\beta_N$ and
  $b_{\rho}^N$ respectively by \eqref{Eq: beta_n} and
  \eqref{Eq:b_n}. By Lemma \ref{lemma: mu_n divergence} and Lemma
  \ref{lemma:b_n divergence} both sequences $\{\beta_N\}$ and
  $\{b_{\rho}^N\}$ diverge.

  Consider now a sequence of functions $\{\varphi_i\}_{i=1}^\infty$ as given by
  Lemma \ref{lemma:eigennotfunctions} taking $\beta=1$ and a set of
  $N$ functions $\{\overline{\varphi}_i\}_{i=1}^N$ given by Lemma
  \ref{lemma:eigendilationfunctions} taking
  $\overline{\beta}= 2\beta_N$ and $\overline{b}=1$.  Moreover, define
  the functions
  \begin{equation*}
    \gamma_{0,N}\colon \mathbb{S}^{N-2} \to H^1_{\mu}(\G) :
    (a_1,\dots,a_{N-1}) \mapsto \sum_{i=1}^{N-1}a_i\varphi_i,
    \hspace{2em}
    \gamma_{1,N}\colon \mathbb{S}^{N-2} \to H^1_{\mu}(\G) :
    (a_1,\dots,a_{N-1}) \mapsto \sum_{i=1}^{N-1}a_i\overline{\varphi}_i.
  \end{equation*}
  which satisfy, for every $N \geq 2$ and $a \in \mathbb{S}^{N-2}$,
  \begin{equation*}
    \begin{cases}
      \|\gamma_{0,N}(a)'\|_{L^2(\G)} = 1, \\[1\jot]
      E_\rho \bigl(\gamma_{0,N}(a), \G) \bigr) = \frac{1}{2},
    \end{cases}
    \hspace{2em}\text{and}\hspace{2em}
    \begin{cases}
      \|\gamma_{1,N}(a)'\|_{L^2(\G)}\ge 2\beta_N , \\[1\jot]
      E_\rho \bigl(\gamma_{1,N}(a), \G) \bigr) \le 1.
    \end{cases}
  \end{equation*}
  From Lemma \ref{Pierotti-Verzini} we know that the set
  \begin{equation*}
    \Gamma_N
    = \bigl\{ \gamma\in C([0,1]\times \mathbb{S}^{N-2},\, H^1_{\mu}(\G))
    \bigm|
    \forall t \in [0,1],\ \gamma(t,\cdot) \text{ is odd, }
    \gamma (0, \cdot) = \gamma_{0,N}, \text{ and }
    \gamma(1, \cdot) = \gamma_{1,N} \bigr\}
  \end{equation*}
  is not empty.

  \smallskip

  Now, we want to use Theorem~\ref{prop:MPcondition} with the choice
  $\Phi = E_\rho(\cdot, \G)$, $d = N-2$,
    $J(u)=\|u'\|_{L^2(\G)}$, $\beta = \beta_N$, and $W = V_{N-2}$.   We easily check that its
  assumptions $(H1)$ and $(H2)$
  are satisfied for any $N$ sufficiently large (uniformly in $\rho$) and thus \eqref{SMP2LL}
  also holds.  Finally, using $b_{\rho}^N\to +\infty$ as $N\to \infty$
  and~\eqref{SMP2L}, 
  we get that $c_{\rho}^N\to+\infty$ as $N\to\infty$.
\end{proof}

\section{Proof of Theorem \ref{thm: main exL}}
\label{section:strong conv}

This section is devoted to the proof of Theorem~\ref{thm: main
  exL}.  As a consequence of Proposition \ref{Prop: inftylevels}, we
may apply Theorem \ref{Objective1} to the family of functionals given
by \eqref{para func}. From Theorem \ref{Objective1} and the
considerations just after it (see in particular
\eqref{free-gradient}--\eqref{def-almost-Lagrange}),
for all $N \in \N$ large enough and for almost every $\rho \in [1/2,1]$, we deduce the existence of a
bounded sequence
$\{u_{\rho,n}^N\}_{n=1}^{\infty} \subset H_{\mu}^1(\G) $, that we
shall simply denote $\{u_n\}$, such that
\begin{equation}\label{const critL}
E_{\rho}(u_n, \G) \to c_{\rho}^N
\end{equation}
and
\begin{equation}\label{const crit}
  E'_{\rho}(u_n, \G) +  \lambda_n (u_n,\cdot) \to 0
  \quad \text{in the dual of }  H^1_{\mu}(\mathcal{G}),
\end{equation}
where
\begin{equation}\label{lambda}
  \lambda_n:= -\frac{1}{\mu} E'_\rho(u_n, \G)[u_n].
\end{equation}
Finally, there exists a sequence $\{\zeta_n\}\subset \mathbb{R}^+$ with $\zeta_n\to 0^+$ such that, if the inequality
\begin{eqnarray}\label{L-Hess crit}
  \int_{\G} |\varphi'|^2 + \left( \lambda_n
  - (p-1)\rho \kappa(x)|u_n|^{p-2}\right) \varphi^2 \intd x
  = E''_\rho(u_n, \G)[\varphi, \varphi]
  + \lambda_n \|\varphi\|_{L^2(\mathcal{G})}^2
  < -\zeta_n \|\varphi\|_{H^1(\G)}^2
\end{eqnarray}
holds for any $\varphi \in W_n \setminus \{0\}$ in a subspace $W_n$ of
$T_{u_n} H^1_{\mu}(\G)$, then the dimension of $W_n$ is at most
$N$.

\smallskip

Since $\{u_n\} \subset H^1(\G)$ is bounded, passing to a
subsequence we may assume that there exists
$u_\rho^N\in H^1(\mathcal{G})$ such that
\begin{align}
  u_{n}\rightharpoonup u_{\rho}^N
  &\quad\text{in } H^1(\mathcal{G}),\label{3-5-1}\\
  u_n\to u_{\rho}^N
  &\quad\text{in } L_{{\rm{loc}}}^r(\mathcal{G})
    \text{ for all } r \ge 2.
    \label{3-5-2}
\end{align}
Observe also that, since $\{u_n\} \subset H^1(\G)$ is a bounded
sequence, it follows from \eqref{lambda} that
$\{\lambda_n\} \subset \R$ is bounded.  As before, passing to a subsequence,
there exists $\lambda_\rho^N\in \mathbb{R}$ such that
$\lim\limits_{n\to+\infty}\lambda_n=\lambda_{\rho}^N$.

\smallskip

The sequences $\{\lambda_{\rho}^N\}_{N =1}^{\infty} \subset \R$ and
$\{u_{\rho}^N \}_{N =1}^{\infty} \subset H^1_{\mu}(\G)$ are the
candidates to prove Theorem \ref{thm: main exL}. We begin by verifying
that the limit $u_\rho^N \in H^1(\G)$ solves \eqref{Eq:pb rho}. Indeed, using
(\ref{const crit}) and the fact that
$\lim\limits_{n\to+\infty}\lambda_n=\lambda_{\rho}^N$, we get
\begin{align}
  0
  &= \lim_{n\to\infty} \bigl( E_{\rho}'(u_n, \G)
    + \lambda_n (u_n, \cdot) \bigr) [\eta] \nonumber\\ 
  &= \lim_{n\to\infty}\left[
    \int_\G u_n'\eta' + \lambda_n \int_\G u_n\eta
    - \rho \int_\K|u_n|^{p-2}u_n\eta\right]\nonumber\\
  &= \int_\G (u_{\rho}^N)'\eta' + \lambda_\rho^N\int_\G u_\rho^N\eta
    - \rho\int_\K|u_\rho^N|^{p-2}u_\rho^N\eta
    \label{Eq:weakurho}
\end{align}
for every $\eta\in H^1(\G)$.
We have thus proved the claim.

We now focus on proving the strong convergence of the sequences $\{u_n\} \subset H^1(\G)$ to ensure that the limits $u_\rho^N$ belong to the mass constraint $H^1_{\mu}(\G)$. 

\begin{proposition}\label{lemma:sconvergenceif}
  The following convergence holds:
  \begin{equation*}
    \int_\G |(u_n-u_\rho^N)'|^2
    + \lambda_\rho^N\int_\G |u_n- u_\rho^N|^2
    \xrightarrow[n \to \infty]{} 0 .
  \end{equation*}
  In particular, if $\lambda_{\rho}^N >0$, the sequence $\{u_n\}$
  converges strongly in $H^1(\G)$.
\end{proposition}
\begin{proof}
  First, rewriting \eqref{const crit} as follows:
  \begin{align*}
    o(1) \|\eta\|_{H^1(\G)}
    &= \int_\G u_n'\eta' - \rho\int_\K|u_n|^{p-2}u_n\eta
      + \lambda_n\int_\G u_n\eta\\
    &= \int_\G u_n'\eta'-\rho\int_\K|u_n|^{p-2}u_n\eta
      + \lambda_\rho^N\int_\G u_n\eta
      + (\lambda_n-\lambda_\rho^N)\int_\G u_n\eta ,
  \end{align*}
  we get
  \begin{equation}\label{Eq:umlimit}
    \int_\G u_n'\eta' - \rho\int_\K|u_n|^{p-2}u_n\eta
    + \lambda_\rho^N\int_\G u_n\eta
    = o(1) \|\eta\|_{H^1(\G)}.
  \end{equation}
  Now, taking the difference between \eqref{Eq:umlimit} and
  \eqref{Eq:weakurho}, choosing $\eta = \eta_n := u_n-u_\rho^N$
  and taking into account \eqref{3-5-2} and  that $\{\eta_n\}$ is bounded, we obtain
  \begin{align*}
    o(1) = o(1)\|\eta_n\|_{H^1(\G)}
    &= \int_\G \bigl(u_n' - (u_{\rho}^N)' \bigr)\eta'_n
      - \rho\int_\K \bigl(|u_n|^{p-2}u_n-|u_\rho^N|^{p-2}u_\rho^N\bigr)\eta_n
      + \lambda_\rho^N \int_\G (u_n- u_\rho^N)\eta_n\\
    &= \int_\G \bigl(u_n'-(u_\rho^N)' \bigr)\eta'_n
      + \lambda_\rho^N\int_\G (u_n- u_\rho^N)\eta_n
      + o(1) \|\eta_n\|_{H^1(\G)}\\
    &= \int_\G \bigl|(u_n-u_\rho^N)'\bigr|^2
      + \lambda_\rho^N\int_\G |u_n- u_\rho^N|^2
      + o(1),
  \end{align*}
  which proves the claim.
\end{proof}

In order to apply Proposition \ref{lemma:sconvergenceif} we need to show that the assumption $\lambda_{\rho}^N >0$ holds. We will do this in two steps. In first place, we will show that $\lambda_{\rho}^N < 0$ is not possible by making use of Lemma \ref{lemma:Louis1}. The following result will aid to check that its assumptions hold.

\begin{lemma}\label{lemma:L-eigenvalue}
  For any $\lambda <0$ and $d \in \N$, there exists a subspace $Y$ of
  $H^1(\mathcal{G})$ with $\dim (Y)= d$ such that
  \begin{equation*}
    E''_{\rho}(u_n, \G)[w,w] + \lambda \|w\|_{L^2(\G)}^2
    = \int_{\mathcal{G}}|w'|^2 \intd x
    + \lambda \int_{\mathcal{G}}|w|^2 \intd x
    \leq \frac{\lambda}{2} \|w\|_{H^1(\G)}^2,
    \qquad \forall \, w \in Y.
  \end{equation*}
\end{lemma}

\begin{proof}
	We proceed similarly to the proof of
    Lemma~\ref{lemma:eigennotfunctions}.
	Take $\varphi \in C^{\infty}_{\text{c}}(\mathbb{R})$ with
	$\supp \varphi \subset \intervaloo{0, 1}$ and such that
	$\int_0^{+\infty} |\varphi|^2 \intd x=1$.
	Viewing $\varphi$ as a function in $H^1(\G)$ whose support
	is contained in a half-line which we identify with
	$\intervalco{0, \infty}$, we define (using the notation of \eqref{Eq:dilation})
	\begin{equation*}
		\varphi_1 := \varphi^\tau,
	\end{equation*}
	where $\tau > 0$ is taken small enough so that
	\begin{equation}
		\label{cond_tau}
		\tau^2 \| \varphi' \|_{L^2(\R)^2} + \lambda
		\le \frac{\lambda}{2} (\tau^2 \| \varphi' \|_{L^2(\R)^2} + 1).
	\end{equation}
	One has that
	\begin{equation*}
		\| \varphi_1 \|_{L^2(\G)} = 1, \qquad
		\| \varphi_1' \|_{L^2(\G)} = \tau^2 \| \varphi' \|_{L^2(\G)}.
	\end{equation*}
	Define now, for $i \geq 2$,
	\begin{equation*}
		\varphi_i(x)
		:= \varphi_1\left(x - \frac{i-1}{\tau} \right).
	\end{equation*}
	Since $\supp(\varphi_i) \subset \bigintervaloo{\frac{i-1}{\tau},
	\frac{i}{\tau}}$, all the $\varphi_i$ have disjoint supports.
	Let $Y\subset H^1(\G)$ be the subspace generated by $\varphi_1, \dotsc, \varphi_d$.
	Any element $w \in Y$ writes as
	\begin{equation*}
		w := \sum_{i=1}^{d} \theta_i\varphi_i.
	\end{equation*}
	where $\theta_1, \dots, \theta_{d} \in \R$.
	By direct calculations we have
	\begin{equation*}
          \begin{split}
            \int_{\mathcal{G}}|w'|^2 \intd x
            + \lambda \int_{\mathcal{G}}|w|^2 \intd x
            &= \tau^2 \Biggl( \sum_{i=1}^{d} \theta_i^2
              \|\varphi'\|_{L^2(\R)}^2 \Biggr)
              + \lambda \Biggl( \sum_{i=1}^{d} \theta_i^2 \Biggr)\\
            &= (\tau^2 \|\varphi'\|_{L^2(\R)}^2 + \lambda) \sum_{i=1}^{d} \theta_i^2.
          \end{split}
	\end{equation*}
	Similarly, $\|w\|_{H^1(\G)}^2 = (\tau^2 \|\varphi'\|_{L^2(\R)}^2 + 1) \sum_{i=1}^{d} \theta_i^2$.
	Therefore, \eqref{cond_tau} implies that
	\begin{equation*}
		\int_{\mathcal{G}}|w'|^2 \intd x
		+ \lambda \int_{\mathcal{G}}|w|^2 \intd x
		= (\tau^2 \|\varphi'\|_{L^2(\R)}^2 + \lambda) \sum_{i=1}^{d} \theta_i^2
		\le \frac{\lambda}{2}  (\tau^2 \|\varphi'\|_{L^2(\R)}^2 + 1) \sum_{i=1}^{d} \theta_i^2
		= \frac{\lambda}{2} \|w\|_{H^1(\G)}^2.
	\end{equation*}
  The fact that $w$ vanishes outside the half-line justifies the
  equality in the claim,
  ending the proof.
\end{proof}

Observe that the codimension of $T_{u_n} H^1_{\mu}(\G)$ in
$H^1(\G)$ is one. Thus, if inequality \eqref{L-Hess crit} holds
for every $\varphi \in W_n \setminus \{0\}$ for a subspace $W_n$ of
$H^1(\G)$, then the dimension of $W_n$ is at most $N+1$.
Let $\lambda < 0$.  Let $Y$ be the space of dimension $d = N+2$
provided by Lemma~\ref{lemma:L-eigenvalue}.
We may thus apply Lemma \ref{lemma:Louis1} to obtain that
\begin{equation}\label{Eq:positivelambda}
  \lambda_\rho^N \ge 0.
\end{equation}
Combining Proposition~\ref{lemma:sconvergenceif} and
\eqref{Eq:positivelambda}, we get that
\begin{equation*}
  \int_\G |(u_n-u_\rho^N)'|^2\to 0.   
\end{equation*}
Using in addition \eqref{3-5-2} and recalling that the
nonlinearity acts only on the compact core $\K$, we obtain
that $E_{\rho}(u_n, \G) \to E_{\rho}(u_{\rho}^N, \G)$. In particular,
in view of \eqref{const critL}, it follows that
\begin{equation}\label{keep_level}
E_{\rho}(u_{\rho}^N, \G) = c_{\rho}^N.
\end{equation}
We will now prove  that $\lambda_{\rho}^N=0$ is not possible either, assuming that $N \in \N$ is large enough uniformly in
  $\rho \in \intervalcc{1/2, 1}$. It is here that we will use
what has been developed in Section \ref{Preliminaries_II}: assume by
contradiction that there exists a subsequence
  $\{u^{N_k}_{\rho_k}\}_{k=1}^\infty$, with $N_k \to +\infty$ and
  $\rho_k \in \intervalcc{1/2, 1}$ for all $k$,
  such that the weak limits
  $u_{\rho_k}^{N_k} \in H^1(\G)$ have an associated
  $\lambda_{\rho_k}^{N_k}$ which
is $0$. By Proposition \ref{Prop: inftylevels},
$c_{\rho}^N \xrightarrow[N\to \infty]{} +\infty$
uniformly w.r.t.\ $\rho$, and thus we have from
\eqref{keep_level} that
$E_{\rho_k}(u_{\rho_k}^{N_k}, \G) \to + \infty$ as $k \to \infty$.
This is
in contradiction with Proposition~\ref{forbid} since
$\{u_{\rho_k}^{N_k}\}_{k=1}^\infty \subset H^1_{\mu}(\G)$.
In conclusion, we have $\lambda_\rho^N>0$.

\bigskip

Finally let us show that the Morse index $\morse(u_{\rho}^N)$ of
$u_{\rho}^N$ as a solution to \eqref{Eq:pb rho} satisfies
$\morse(u_{\rho}^N) \leq N +1$.  We recall that the Morse index of a
solution $u \in H^1(\G)$ of \eqref{1.2LL} is defined as the maximal
dimension of a subspace $W \subset H^1(\G)$ such that
$Q(\varphi;u, \G) <0$ for all $\varphi \in W \setminus \{0\}$, where
\begin{equation*}
  Q(\varphi;u, \G)
  := \int_{\G} |\varphi'|^2
  + \bigl(\lambda- \kappa(x)(p-1)\rho|u|^{p-2} \bigr) \varphi^2 \intd x.
\end{equation*}
We also note the relationship between the Morse index of a solution to \eqref{1.2LL} and the Morse index as a constrained critical point (refer to Definition \ref{def: app morse}) via the equality
\begin{align}\label{egalite}
  D^2 E_{\rho}(u_{\rho}^N, \G)[w,w]
  &:= E''_{\rho}(u_{\rho}^N, \G)[w,w] +\lambda_{\rho}^N (w,w)
    \nonumber \\
  &= \int_{\G} \left[ |w'|^2 + 
    \bigl( \lambda_{\rho}^N -(p-1)\kappa(x) |u_{\rho}^N|^{p-2}\bigr)
    w^2\right]  \intd x,
    \quad \text{for all } w \in H^1(\G).
\end{align}
Since $u_{\rho,n}^N \to u_{\rho}^N$ as $n \to \infty$, we know
from Remark~\ref{strong-convergence} that the Morse index of
$u_{\rho}^N \in H^1_{\mu}(\G)$ as a constrained critical point is less
than $N$. In view of \eqref{egalite} and of the fact that
$H_{\mu}^1(\G)$ is of codimension 1 in $H^1(\G)$ we deduce
\begin{equation}\label{Eq:good morse index}
  \morse(u_{\rho}^N) \leq N+1.
\end{equation}
Summarizing what has been observed so far we can give the

\begin{proof}[Proof of Theorem \ref{thm: main exL}]
  For any $\mu >0$ and any $N \in \N$ sufficiently large, we have
  shown that the particular bounded Palais-Smale sequence, satisfying
  \eqref{const critL}--\eqref{L-Hess crit}, provided for almost every
  $\rho \in [1/2, 1]$ by the application of Theorem \ref{Objective1}
  is converging. This leads to the existence of sequence of couples
  $\{(\lambda_\rho^N,u_\rho^N)\} \subset (0, + \infty) \times
  H^1_{\mu}(\G)$ which are solutions to \eqref{Eq:pb rho}. We also
  have by \eqref{keep_level} that
  $E(u_{\rho}^N, \G) = c_{\rho}^N \to + \infty$.  The estimate
  \eqref{Eq:good morse index} completes the proof.
\end{proof}


\section{Proof of Theorem \ref{thm: Whole ex}}\label{Proof_Th1}

Let $\mu >0$ and $N \in \N$ be sufficiently large.  By Theorem
\ref{thm: main exL}, it is possible to choose a sequence
$\rho_n \to 1^-$, and a corresponding sequence of critical points
$u_{\rho_n}^N \in H^1_\mu(\G)$ of $E_{\rho_n}(\cdot\,, \G)$
constrained to $H^1_\mu(\G)$, at the level $c_{\rho_n}^N$ and having a
Morse index $\morse(u_{\rho_n}^N) \le N+1$.  Additionally, the
Lagrange multipliers satisfy $\lambda_{\rho_n}^N >0$.

\smallskip

To prove Theorem \ref{thm: Whole ex}, it clearly suffices to show that
$\{u_{\rho_n}^N\} \subset H^1_{\mu}(\G)$ converges. For this the key
point is to show that $\{u_{\rho_n}^N\} \subset H^1(\G)$ is bounded.
The monotonicity of $c_{\rho}^N$, as a function of
$\rho \in \left[1/2, 1 \right]$ implies that $\{c_{\rho_n}^N\}$ is
bounded as it belongs to $\intervalcc{c_1^N, c_{1/2}^N}$
with $c_1^N, c_{1/2}^N \in \R$ (see Remark \ref{c_rho_N_finite}).
In addition, since, thanks to the Kirchhoff boundary
condition
\begin{equation*}
  \int_{\G} \bigl|(u_{\rho_n}^N)'\bigr|^2
  + \lambda_{\rho_n}^N \bigl(u_{\rho_n}^N\bigr)^2 \intd x
  = \rho_n \int_{\mathcal{K}} \bigl|u_{\rho_n}^N\bigr|^p\intd x,  
\end{equation*}
it follows that
\begin{equation*}
  c_{\rho_n}^N
  = E_{\rho_n}(u_{\rho_n}^N, \G)
  = \left( \frac12-\frac1p\right) \int_{\G} |(u_{\rho_n}^N)'|^2\intd x
  - \frac{\lambda_{\rho_n}^N \mu}{p}.  
\end{equation*}
Therefore
\begin{equation*}
\left( \frac12-\frac1p\right) \int_{\G} |(u_{\rho_n}^N)'|^2\intd x = c_{\rho_n}^N + \frac{\lambda_{\rho_n}^N  \mu}p
\end{equation*}
and thus, if
$\{\lambda_{\rho_n}^N\} \subset \intervaloo{0, +\infty}$ is
bounded, then $\{u_{\rho_n}^N\} \subset H^1(\G)$ is bounded as
well. At this point to conclude the proof of Theorem \ref{thm: Whole
  ex} we just need to make use of the following result which is
\cite[Corollary~1.4]{CaGaJeTr} adapted to our notation.

\begin{lemma}\label{info-norms}
  Let $\G$ be a metric graph satisfying Assumption \eqref{HG}, and
  $p>6$. Assume that $(\rho_n) \subseteq \intervalcc{\frac{1}{2}, 1}$
  is a sequence converging to $1$. Let
  $\{(\lambda_n, u_n)\} \subseteq \R \times H^1(\G)$ be a
    sequence of solutions to
  \begin{equation*}
    \begin{cases}
      -u''+\lambda u= \rho\kappa(x)|u|^{p-2}u
      &\text{on every edge}~ \edge \in \mathcal{E},\\[1.5\jot]
      \displaystyle
      \sum\limits_{\edge \incident \vv}u_{\edge}'(\vv)=0
      &\text{at every vertex } \vv \in \mathcal{V},
    \end{cases}
  \end{equation*}
  and satisfy additionally, for some $\mu>0$,
  \begin{equation*}
    \int_{\G}|u_n|^2 \intd x = \mu,
    \quad \text{for all } n \in \N
  \end{equation*}
  and whose Morse indices $\morse(u_n)$ are bounded.  Then,
  the sequence $\{\lambda_n\} \subset \R$ is bounded from above.
\end{lemma}

\medskip

{\small \noindent Acknowledgements:
This work has been carried out in the framework of the Project NQG (ANR-23-CE40-0005-01), funded by the French National Research Agency (ANR).
P. Carrillo, D. Galant and L. Jeanjean thank the ANR for its support. D. Galant is an F.R.S.-FNRS Research Fellow.}

\medskip

{\small \noindent Statements and Declarations: The authors have no relevant financial or non-financial interests to disclose.}

{\small \noindent Data availability: Data sharing is not applicable to this article as no datasets were generated or analysed during the current study.}

\renewcommand{\bibname}{References}
\bibliographystyle{plain}
\bibliography{multiplicity_loc_nonlin}
\vspace{0.25cm}

\end{document}